\documentclass[a4paper,10pt,reqno]{amsart}

\usepackage[english,activeacute]{babel}
\usepackage[utf8]{inputenc}
\usepackage[foot]{amsaddr}
\usepackage{hyperref}
\usepackage{cleveref}
\usepackage{subcaption}

\usepackage{enumitem}
\setlist[itemize]{topsep=0.2em, itemsep=0.2em, leftmargin=2em}
\setlist[enumerate]{topsep=0.2em, itemsep=0.2em, leftmargin=2em}
\setlist[description]{topsep=0.2em, itemsep=0.2em, leftmargin=2em}

\usepackage{latexsym}

\usepackage{graphicx}
\usepackage{color}
\usepackage{tikz-cd}
\usetikzlibrary{matrix}

\usepackage[a4paper,top=3.3cm,bottom=3cm,left=2cm,right=2cm,bindingoffset=5mm]{geometry}
\usepackage{marginnote}
\definecolor{alert}{rgb}{0.8,0,0.3}
\newcommand{\alert}[1]{%
	\marginpar{%
		\ifodd\value{page} \raggedright \else \raggedleft \fi
		\footnotesize{\textcolor{alert}{#1}}
	}
}

\newcommand{\N}{\mathbb{N}}

\newcommand{\R}{\mathbb{R}}

\renewcommand{\H}{\mathbb{H}}

\newcommand{\E}{\mathbb{E}}
\newcommand{\Sf}{\mathbb{S}}

\newcommand{\csch}{\mathop{\rm csch}\nolimits}
\newcommand{\arctanh}{\mathop{\rm arctanh}\nolimits}

\newcommand{\Div}{\mathop{\rm div}\nolimits}

\newcommand{\dist}{\mathop{\rm dist}\nolimits}
\newcommand{\Length}{\mathop{\rm Length}\nolimits}

\newcommand{\Flux}[2]{\mathrm{Flux}_{#1}\left(#2\right)}

\newcommand{\Prod}[1]{ \left\langle {#1} \right\rangle}

\newcommand{\sign}[1]{\mathop{\rm sign}\nolimits(#1)}

\newtheorem{theorem}{Theorem}[section]
\newtheorem{proposition}[theorem]{Proposition}

\newtheorem{lemma}[theorem]{Lemma}

\theoremstyle{definition}
\newtheorem{definition}[theorem]{Definition}

\theoremstyle{remark}
\newtheorem{remark}[theorem]{Remark}

\numberwithin{equation}{section}
\begin{document}
	
	\title[Constant mean curvature graphs with prescribed asymptotic values in $\E(-1,\tau)$]{Constant mean curvature graphs with prescribed asymptotic values in $\E(-1,\tau)$}
	
	\author{Andrea Del Prete}
	\address{Dipartimento di Matematica "Felice Casorati"\\
		Universit\`{a} degli Studi di Pavia\\
		Via Ferrata 5 -- 27100 Pavia (Italy)}
	\email{andrea.delprete@unipv.it}
	
	\thanks{The author was partially supported by ``INdAM - GNSAGA Project", codice CUP\_E55F22000270001, PRIN project “Geometry and topology of manifolds” project no. F53D23002800001 and by MCIN/AEI project PID2022-142559NB-I00. 
	}
	
	
	
	\begin{abstract}
		In the homogeneous manifold $\E(-1,\tau),$ for $-\tfrac{1}{2}<H<\tfrac{1}{2},$ {we define a new product compactification in which the slices $\left\{t=c\right\}_{c\in\R}$ are rotational $H$-surfaces. This product compatification is the natural setting where it makes sense to study the asymptotic Dirichlet Problem for the constant mean curvature equation. Indeed, for every rectifiable curve $\Gamma$ projecting bijectively onto $\partial\H^2$ we prove the existence of a unique entire $H$-graph that is asymptotic to $\Gamma$.} We also find necessary and sufficient conditions for the existence of $H$-graphs over unbounded domains having prescribed, possibly infinite boundary data.
	\end{abstract}
	
	\maketitle
	
	\section{Introduction}\label{sec:intro}\
	
	Constant mean curvature surfaces in simply connected homogeneous 3-manifolds have been a subject of extensive research by various authors over the past two decades. In particular, considerable attention has been devoted to 3-manifolds with an isometry group of dimension at least 4. These 3-manifolds, excluding the hyperbolic space $\H^3$, can be classified in a 2-parameter family denoted as $\E(\kappa, \tau)$, where $\kappa$ and $\tau$ are real numbers. This classification depends on the property of the $\E(\kappa, \tau)$-spaces to admit a Riemannian submersion $\pi\colon\E(\kappa,\tau)\to\mathbb{M}(\kappa)$ over the complete simply connected Riemannian surface with constant curvature $\kappa$ whose fibers are the integral curves of the unique unitary Killing vector field $\xi\in\mathfrak{X}(\E)$ such that $\tau$ is the bundle curvature (that is $\nabla_X\xi=\tau X\wedge\xi$ for all vector fields $X\in\mathfrak{X}(\E), $ where $\nabla$ is the Levi-Civita connection of $\E$ and $\wedge$ is the cross product in $\E$ with respect to a chosen orientation, see \cite{Man14}).
	
	This geometric setup naturally prompts the question of investigating constant mean curvature graphs over specific domains of $\mathbb{M}(\kappa)$, i.e. sections of the Riemannian submersion obtained by deforming a fixed zero section through the flow lines.
	
	In this paper we focus on the case $\kappa=-1$ and investigate graphs with constant mean curvature $H$ that will be referred as an $H$-graphs or, whenever $H$ is not desired to be explicit, CMC graphs. 
	
	One of the pioneering works in this setting is due to Nelli and Rosenberg in $\H^2\times\R=\E(-1,0)$ (see \cite{NelRos02}), who solved the following problems.
	\begin{itemize}
		\item The existence of a unique entire minimal graph with arbitrarily prescribed asymptotic values{, that is, the graph is asymptotic to a rectifiable curve that projects bijectively onto $\partial \H^2$.}
		\item The Jenkins--Serrin problem over relatively compact domains, that is, to find necessary and sufficient conditions that guarantee the existence and uniqueness of a solution for the Dirichlet problem for the minimal surface equation with possibly infinite boundary data.
	\end{itemize}
	
	So far, no result with an analogous statement as in \cite{NelRos02} has been given for CMC graphs with prescribed asymptotic boundary value. This is also due to the fact that in the standard compactification for $\E(-1,\tau)$, i.e. the model described in Section~\ref{sec:theSpace}, such a result cannot be achieved due to the Slab Theorem (see \cite{HauMenRod19}), which implies that any entire function satisfying the constant mean curvature equation must diverge.
	
	Concerning the Jenkins--Serrin problem, several extensions have been achieved.
	A Jenkins--Serrin problem for CMC graphs in $\H^2\times\R$ over relatively compact domains of $\H^2$ was proved by Hauswirth, Rosenberg and Spruck in \cite{HauRoSpr09} and by Eichmair and Metzger in \cite{EM} in a general product space $M\times\R$. Furthermore, in \cite{DMN}, the author, together with Manzano and Nelli, proves an analogous result for minimal graphs in a general Killing submersion, which include the space $\E(-1,\tau).$
	In \cite{MazRodRos11} Mazet, Rodriguez and Rosenberg proved in $\H^2\times\R$ a Jenkins--Serrin type result for minimal surfaces over unbounded domains of $\H^2$, allowing the asymptotic boundary of such domains to contain open subsets of $\partial_\infty\H^2$. Later, Melo proved the same result in the space $\E(-1,\tau)$ \cite{Mel}.
	Finally, in \cite{FolMe12}, Folha and Melo in \cite{FolMe12} solved the Jenkins--Serrin problem for $H$-graphs in $\H^2\times\R$ with $0<H<1/2$ over an unbounded domain whose asymptotic boundary does not contain open subsets of $\partial_\infty\H^2$.
	
	In this paper we consider a new model for $\E(-1,\tau)$ in which the zero section has a fixed constant mean curvature $H\in\left(-\tfrac{1}{2},\tfrac{1}{2}\right)$. Using this model, we can consider a new product compactification (which we call $H$-compactification) of $\E(-1,\tau)$ such that every family of $H$-surfaces that foliates the space $\E(-1,\tau)$ also foliates its asymptotic boundary $\pi^{-1}(\partial_\infty\H^2)$, as it happens for the standard product compatification when we consider minimal surfaces (see \Cref{rem:Hcomp} and \Cref{Hcompactification} for a detailed description of the relation between the classical product compactification and the $H$-compactification). Furthermore, considering $|H|$ to be strictly less then the critical value for the mean curvature ensures that the fibers of the submersion stay transversal to the foliation at the infinity reproducing the same properties of the classical product compactification in the minimal case (see \Cref{sec:newModel} for the details).
	
	In particular, fixed $H\in\left(-\tfrac{1}{2},\tfrac{1}{2}\right)$, we solve the following problems in the $H$-compatification of $\E(-1,\tau)$:
	\begin{itemize}
		\item Given a rectifiable curve $\Gamma\in\partial_\infty\H^2\times^H\R$ that is the graph of a function of $\partial_\infty\H^2$, we are able to construct an $H$-graph $\Sigma$ that is asymptotic to $\Gamma$ (see \Cref{thm:entire}). In such a way, we have a direct control over the symmetries of the graph $\Sigma$. Indeed, by means of the Maximum Principle one can prove that $\Sigma$ inherits the symmetries of the curve $\Gamma$. We cannot cover the case $H=1/2$ because of the half-space property given in \cite{NS,Maz15}. Moreover, due to our result, a half-space theorem for entire $H$-graphs with $H\in\left(-\tfrac{1}{2},\tfrac{1}{2}\right)$ cannot be proven.
		\item We give the necessary and sufficient conditions that guaranty the existence of a solution to the Jenkins-Serrin problem for $H$-graphs over a larger set of unbounded domains of $\H^2$, allowing them to contain parts of $\partial\H^2$ in their asymptotic boundary (see \Cref{thm:JS}).
	\end{itemize}
	 {Usually, the mean curvature is considered to be non-negative since, when $\tau=0$, we can solve the case $H\in(-1/2,0)$ simply by applying a reflection with respect to the horizontal totally geodesic graph to a solution with positive $H$. In this article, we describe the new model and the upper and lower barriers in this setting for both $H>0,$ and $H<0$. We then prove the main theorems under the assumption $H\in(0,1/2)$, as the other case follows by an analogous argument.}
	
	The paper is organized as follows.
	\Cref{sec:theSpace} is a description of the basic properties of the space $\E(-1,\tau)$ and its compactification. In \Cref{sec:asymptotic-behaviour} we recall the constant mean curvature graphs that are invariant with respect to a one-parameter group of isometries of $\E(-1,\tau)$ studied by Peñafiel in \cite{Pen12} and we study their asymptotic behaviour.
	In \Cref{sec:newModel} we introduce the new model for $\E(-1,\tau)$ in which the plane $\left\{t=0\right\}$ has constant mean curvature $H\in(-1/2,1/2)$.
	In \Cref{sec:Barriers}, we use the surfaces studied in \Cref{sec:asymptotic-behaviour} to construct barriers for the Dirichlet Problem for the constant mean curvature equation. Finally, the main results of this paper are given in \Cref{sec:EntireHGraphs,App}, where, respectively, we prove the existence of entire $H$-graphs with prescribed asymptotic boundary and the Jenkins--Serrin problem.
	
	\section{The space $\E(-1,\tau)$}\label{sec:theSpace}\
	The space $\E(-1,\tau)$ is the homogeneous 3-dimensional manifold defined as $\Omega\times\R$, endowed with the metric
	\begin{equation}\label{eq:metricGen}
		ds^2_\lambda=\lambda^2\left(dx^2+dy^2\right)+\left[dt-2\tau\lambda\left(\left(\tfrac{1}{\lambda}\right)_ydx-\left(\tfrac{1}{\lambda}\right)_xdy\right)\right]^2.
	\end{equation}
	where $\Omega=\left\{(x,y)\in\R^2\mid\lambda(x,y)>0\right\}$ and $\lambda\in\mathcal{C}^\infty(\Omega)$ satisfies $\Delta(\log(\lambda))-\lambda^2=0$.
	
	Notice that $\left(\Omega,\lambda^2(dx^2+dy^2)\right)$ is isometric to the hyperbolic plane $\mathbb{H}^2$ with a constant sectional curvature of $-1$ and the Killing submersion is $\pi(x,y,t)=(x,y)$ with $\partial_t$ being the unitary Killing field satisfying the identity $\nabla_X\partial_t=\tau X\wedge\partial_t$ for any $X\in\mathfrak{X}(\E)$. In this context, we call \emph{vertical} every vector field parallel to $\partial_t,$ and horizontal every vector field orthogonal to $\partial_t$. 
	In this model, the mean curvature of the graph $\Sigma_u=(x,y,u(x,y))\subset\Omega\times\R$ of the function $u\in\mathcal{C}^2(\Omega)$ is given in the divergent form as
	\begin{equation}\label{eq:meancurvature}
		2H=\Div\left(\frac{ Gu}{\sqrt{1+\|Gu\|^2}}\right),
	\end{equation}
	where $\Div(\cdot)$ and $\|\cdot\|$ are the divergence and norm of $\mathbb{H}^2$ in the conformal metric $\lambda^2(dx^2+dy^2),$ and $Gu=\left(u_x-\tfrac{2\tau \lambda_y}{\lambda}\right)\frac{\partial_x}{\lambda^2}+\left(u_y+\tfrac{2\tau \lambda_x}{\lambda}\right)\frac{\partial_y}{\lambda^2}$ is the generalized gradient. 		
	{\begin{remark}
			Equation \eqref{eq:meancurvature} follows from $2H=\Div_{\E(-1,\tau)}(N)=\Div_{\H^2}(\pi_* N)$, where $\pi\colon\E(-1,\tau)\to\H^2$ is the projection on the first two coordinates and $N_u=\frac{\partial_t+Gu}{\sqrt{1+\| Gu\|^2}}$ is upward-pointing unit normal to the Killing graph of $u$.
	\end{remark}}
	
	Depending on the choice of $\lambda,$ we have two classical models to describe $\E(-1,\tau):$
	\begin{itemize}
		\item  If $\lambda(x,y)=\lambda_{\mathcal{H}}(x,y)=\frac{1}{y}$, we get the half-plane model for $\mathbb{H}^2$ and the half-space model $\mathcal{H}$ for $\E(-1,\tau),$ where the vertical boundary is $(\left\{y=0\right\}\cup\left\{\infty\right\})\times\R \equiv(\R\cup\left\{\infty\right\})\times\R.$ 
		\item If $\lambda(x,y)=\lambda_{\mathcal{C}}(x,y)=\frac{2}{1-x^2-y^2}$, we obtain the Poincar\'e Disk model for $\mathbb{H}^2$ and the cylindrical model $\mathcal{C}$ for $\E(-1,\tau),$ with $\mathbb{S}(1)\times\R$ as vertical boundary.
	\end{itemize}
	
	Using the complex coordinate $z=x+i y$, the isometries between this two models are given by lifting the Möbius transformation of $\mathbb{H}^2$ (see \cite{Cas22} for more details):
	\begin{equation*}\label{iso:HtoC}
		\begin{array}{rccc}
			\psi\colon&\overline{\mathcal{H}}&\to&\overline{\mathcal{C}}\\
			&(x,y,t)&\mapsto&\left(\frac{z-i}{z+1},t-4\tau\arctan\left(\frac{x}{y+1}\right)\right),
		\end{array}
	\end{equation*}
	and
	\begin{equation}\label{iso:CtoH}
		\begin{array}{rccc}
			\psi^{-1}\colon&\overline{\mathcal{C}}&\to&\overline{\mathcal{H}}\\
			&(x,y,t)&\mapsto&\left(\frac{i(z+1)}{1-z},t-4\tau\arctan\left(\frac{y}{1-x}\right)\right).
		\end{array}
	\end{equation}
	
	Some isometries of $\E(-1,\tau),$ can be easily described using either $\mathcal{C}$ or $\mathcal{H}:$
	\begin{itemize}
		\item Vertical translation: in both models  \[\varphi_1\colon(x,y,t)\mapsto(x,y,t+c),\,\forall c\in\R\]
		\item Hyperbolic translation along an horizontal geodesic: using the half-space model $\mathcal{H},$ the hyperbolic translation along the horizontal geodesics in $\left\{x=x_0\right\}$ is given by the map \[\varphi_2\colon(x,y,t)\mapsto(x_0+c(x-x_0),c y ,t),\,\forall c\in\R.\]
		\item Parabolic translation: using the half-space model $\mathcal{H},$ the hyperbolic translation along the horocycle $\left\{y=y_0\right\}$ is given by the map \[\varphi_3\colon(x,y,t)\mapsto(x+c,y,t),\,\forall c\in\R.\]
		\item Rotation: in the cylinder model $\mathcal{C},$ rotations with respect to the $t$-axis are Euclidean rotations with respect to the $t$-axis, that are, \[\varphi_4\colon(x,y,z)\mapsto(x\cos(\theta)-y\sin(\theta),x\sin(\theta)+y\cos(\theta),t),\,\forall\theta\in\R.\]
	\end{itemize}
	
	In the cylinder model $\mathcal{C}$, the surface $\left\{t=0\right\}$ is the rotational minimal surface $\mathcal{U}$ known as umbrella; in the half-space model $\mathcal{H}$, the surface $\left\{t=0\right\}$ is the minimal surface, $\mathcal{S},$ invariant with respect to hyperbolic and parabolic translation.
	
	\textbf{Acknowledgement:} The author thanks the anonymous referees for their helpful comments that improved the quality of the manuscript.
	
	\section{Constant mean curvature surfaces invariant by a one-parameter group of isometries and their asymptotic behaviour}\label{sec:asymptotic-behaviour}
	
	In this section, we provide a complete classification for CMC surfaces that are invariant under either rotations or hyperbolic translations and are vertical graphs of functions of $\Omega$, in either the cylinder or the half-space model. The study of such surfaces has been conducted previously in \cite{SaETou05} and \cite{SaE08} for $\tau=0$ and in \cite{CuiMafPe15} and \cite{Pen12} for $\tau\neq0$.  In particular, our interest focus on the asymptotic behaviour of these invariant surfaces with respect to the surface $\left\{t=0\right\}$ in each model, and the relation between $\mathcal{S}$ and $\mathcal{U}$.
	
	To describe the surfaces invariant by the rotation $\varphi_4$ we will use the cylinder model $\mathcal{C}$ and the geodesic polar coordinates 
	\begin{equation*}
		\begin{cases}
			x(\rho,\theta)=\tanh\left(\tfrac{\rho}{2}\right)\cos(\theta),\\
			y(\rho,\theta)=\tanh\left(\tfrac{\rho}{2}\right)\sin(\theta),
		\end{cases}
	\end{equation*}
	of the base $\mathbb{H}^2$ of the Killing submersion. 
	So, any $H$-surface invariant by $\varphi_4$ is parametrized by
	\[\mathcal{R}^H_d(\rho,\theta)=\left(\tanh\left(\tfrac{\rho}{2}\right)\cos(\theta),\tanh\left(\tfrac{\rho}{2}\right)\sin(\theta),v^H_d(\rho)\right),\]
	where $v^H_d$ is any solution of the differential equation
	\begin{align*}
		\left(\coth(\rho)-16\tau^2\csch^3(\rho)\sinh^4\left(\tfrac{\rho}{2}\right)+4\tau^2\coth(\rho)\tanh^2\left(\tfrac{\rho}{2}\right)+\coth(\rho)\dot{v}^H_d(\rho)\right)\dot{v}^H_d(\rho)^2\\
		+\left(1+4\tau^2\tanh^2\left(\tfrac{\rho}{2}\right)\right)\overset{..}{v}^H_d(\rho)=2H\left(1+4\tau^2\tanh^2\left(\tfrac{\rho}{2}\right)+\dot{v}^H_d(\rho)^2\right)^{\tfrac{3}{2}}.
	\end{align*}
	As it is shown in \cite{Pen12}, the solution is given by the one-parameter family of integral functions
	\begin{equation}\label{eq:rotationalCMCgraphs}
		v^H_d(\rho)=\int_*^\rho\frac{(2H\cosh(r)+d)\sqrt{1+4\tau^2\tanh^2\left(\tfrac{r}{2}\right)}}{\sqrt{\sinh^2(r)-(2H\cosh(r)+d)^2}}dr,
	\end{equation}
	and,  depending on the choice of $d\in\R,$ we can obtain three type of surfaces:
	\begin{itemize}
		\item if $d>-2H,$ then $\mathcal{R}^H_d(\rho,\theta)$ is the upper half of a proper embedded annulus symmetric with respect to the slice $\left\{t=0\right\}$ (see \Cref{RotSurfA});
		\item if $d=-2H,$ then $\mathcal{R}^H_d(\rho,\theta)$ is an entire vertical graph contained in $\left\{t\geq0\right\}$ and tangent to $\left\{t=0\right\}$ at the point $(0,0,0)$ (see \Cref{RotSurfB});
		\item if $d<-2H,$ then $\mathcal{R}^H_d(\rho,\theta)$ is the half of a properly immersed (and non-embedded) annulus, symmetric with respect to slice $\left\{t=0\right\}$ (see \Cref{RotSurfC}).
	\end{itemize}
	
	\begin{figure}[h!]
		\centering
		\begin{subfigure}[b]{0.2\linewidth}
			\centering
			\includegraphics[scale=0.5]{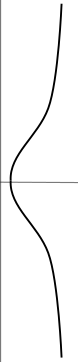}
			\caption{Embedded $H$-annulus}
			\label{RotSurfA}
		\end{subfigure}
		\begin{subfigure}[b]{0.5\linewidth}
			\centering
			\includegraphics[scale=0.5]{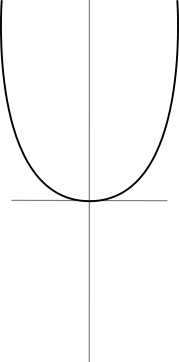}
			\caption{Entire rotational $H$-graph}
			\label{RotSurfB}
		\end{subfigure}
		\begin{subfigure}[b]{0.2\linewidth}
			\centering
			\includegraphics[scale=0.5]{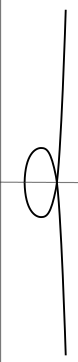}
			\caption{Immersed $H$-annulus}
			\label{RotSurfC}
		\end{subfigure}
		\caption{Generating curves of rotational $H$-surfaces.}
		\label{RotSurf}
	\end{figure}
	
	To study the hyperbolic invariant CMC surfaces, we consider the half-space model $\mathcal{H}$ and consider the hyperbolic translation $\varphi_2$ with respect to the horizontal geodesic $\left\{x=x_0\right\}$. These surfaces are vertical graphs parametrized by $\Theta_d^H(x,y)=\left(x,y,u_d^H\left(\tfrac{x}{y}\right)\right)$ (as it is done in \cite{Cas22} for $H=0$). If we denote by $s=\tfrac{x}{y}$, and assume that the normal points upwise when $H>0$ and downwise when $H<0$ one get that $w_d(s)=\dot{u}_d^H(s)$ satisfies the following equation:
	\begin{equation}\label{eq:hyperbolicCMCgraphs}
		\frac{\left[s^2(4\tau^2+1)+1\right]w_d'(s)+sw_d(s)\left[(s^2+1)w_d(s)^2-6\tau w_d(s)+8\tau^2+2\right]}{2\left[(s^2+1)w(s)^2-4\tau w_d(s)+4\tau^2+1\right]^{\tfrac{3}{2}}}=H.
	\end{equation}
	Integrating this equation we get that one-parameter family of solutions
	\begin{equation}\label{eq:solutioneqHUp}
		w_d(s)=\frac{2\tau}{s^2+1}-\sign{H}\frac{\sqrt{4s^2\tau^2+s^2+1}(d-2H)}{(s^2+1)\sqrt{1-(d-2Hs+s)(d-2Hs-s)}}.
	\end{equation}
	Depending on the number of the real roots of the polynomial $P(s)=1-(d-2Hs+s)(d-2Hs-s)$ we can distinguish three cases:	
	\begin{enumerate}[label=\roman*)]
		\item Case $-\sqrt{1-4H^2}<d<\sqrt{1-4H^2}$.
		
		In this case $P(s)$ has no real roots, so $w_d(s)$, and hence $u_d^H(s)$ is defined for all $s\in\R$. In particular, up to vertical translation we get 
		$$u_d^H(s)=2\tau\arctan(s)-\sign{H}\int_0^s\frac{\sqrt{4t^2\tau^2+t^2+1}(d-2Ht)}{(t^2+1)\sqrt{1-(d-2Ht+t)(d-2Ht-t)}}dt.$$
		Studying the asymptotic behaviour, we get that when $s\to\pm\infty$ $$w_d(s)\approx\sign{H}\left(\frac{2H(1+4\tau^2)}{\sqrt{1-4H^2}s}-\frac{d\sqrt{1+4\tau^2}}{(1-4H^2)^{\tfrac{3}{2}}s^2}+o(s^2)\right)$$ 
		and integrating we get
		\begin{equation}\label{eq:asymptoticbehaviourhyperbolicsolutions}
			u_d^H(s)\approx \sign{H}\left( \frac{2H(1+4\tau^2)}{\sqrt{1-4H^2}}\log(s)+\frac{d\sqrt{1+4\tau^2}}{(1-4H^2)^{\tfrac{3}{2}}s}+o(s^2)\right).
		\end{equation}
		In particular, $\lim_{s\to\pm\infty}v_d^H(s)=\sign{H}\infty$.

		\begin{figure}[h!]
			\begin{subfigure}[b]{0.4\linewidth}
				\begin{tikzpicture}[scale=0.6]
					\draw[->] (-4,0)--(-2,0) node[anchor=south] {$+\infty$}--(2,0) node[anchor=south] {$+\infty$}--(4,0) node[anchor=north] {$x$ axis};
					\draw[gray, very thin,->] (0,0)--(0,4) node[anchor=west] {$y$ axis};
					\node at (1,2) {$u_{d}^H$};
					\node at (-1.5,3.5) {$H>0$};
				\end{tikzpicture}
			\end{subfigure}
			\begin{subfigure}[b]{0.4\linewidth}
				\begin{tikzpicture}[scale=0.6]
					\draw[->] (-4,0)--(-2,0) node[anchor=south] {$-\infty$}--(2,0) node[anchor=south] {$-\infty$}--(4,0) node[anchor=north] {$x$ axis};
					\draw[gray, very thin,->] (0,0)--(0,4) node[anchor=west] {$y$ axis};
					\node at (1,2) {$u_{d}^H$};
					\node at (-1.5,3.5) {$H<0$};
				\end{tikzpicture}
			\end{subfigure}
			\caption{Entire $H$-graphs invariant by hyperbolic translations.}
		\end{figure}
		
		\item Case $d=\pm\sqrt{1-4H^2}$.
		
		In this case $P(s)$ has a unique root $s_0=\tfrac{2H}{d}$ and then, up to vertical translation, we have two solutions:
		$$u_{d,+}^H(s)=2\tau\arctan(s)-\sign{H}\int_{s_0+1}^{s}\frac{\sqrt{4t^2\tau^2+t^2+1}(d-2Ht)}{(t^2+1)\sqrt{1-(d-2Ht+t)(d-2Ht-t)}}dt,$$
		and
		$$u_{d,-}^H(s)=2\tau\arctan(s)-\sign{H}\int_{s}^{s_0-1}\frac{\sqrt{4t^2\tau^2+t^2+1}(d-2Ht)}{(t^2+1)\sqrt{1-(d-2Ht+t)(d-2Ht-t)}}dt.$$
		As in \eqref{eq:asymptoticbehaviourhyperbolicsolutions}, we can prove that $v_{d,\pm}^H(s)$ goes to $+\infty$ as $s$ goes to $\pm\infty$, while $$w_d(s)\approx\sign{H}\left(\frac{\sqrt{1+16\tau^2H^2}}{\sqrt{1-4H^2}}\frac{1}{s_0-s}\right)$$ when $s$ is near to $s_0$. In particular, one can distinguish two cases:
		\begin{enumerate}
			\item If $d=\sqrt{1-4H^2}$, then $$\lim_{s\to s_0^+}w_d(s)=-\sign{H}\infty,\quad \lim_{s\to s_0^-}w_d(s)=\sign{H}\infty,$$ and additionally $$\lim_{s\to s_0^-}u_{d,-}^H(s)=\sign{H}\infty, \qquad \lim_{s\to s_0^+}u_{d,+}^H(s)=-\sign{H}\infty.$$
			
			\begin{figure}[h!]
				\begin{subfigure}[b]{0.4\linewidth}
					\begin{tikzpicture}[scale=0.6]
						\draw[->] (-4,0)--(-2,0) node[anchor=south] {$+\infty$}--(2,0) node[anchor=south] {$+\infty$}--(4,0) node[anchor=north] {$x$ axis};
						\draw (0,0)-- (1.35,1.3) node[anchor=west] {$-\infty$} node[anchor=east] {$+\infty$} --(2.7,2.6) node[anchor=south] {$\tfrac{x}{y}=\frac{2H}{\sqrt{1-4H^2}}$} ;
						\node at (3.2,1) {$u_{d,+}^H$};
						\node at (-1,1.5) {$u_{d,-}^H$};
						\node at (-1,3) {$H>0$};
					\end{tikzpicture}
				\end{subfigure}
				\begin{subfigure}[b]{0.4\linewidth}
					\begin{tikzpicture}[scale=0.6]
						\draw[->] (-4,0)--(-2,0) node[anchor=south] {$-\infty$}--(2,0) node[anchor=south] {$-\infty$}--(4,0) node[anchor=north] {$x$ axis};
						\draw (0,0)-- (1.35,1.3) node[anchor=west] {$+\infty$} node[anchor=east] {$-\infty$} --(2.7,2.6) node[anchor=south] {$\tfrac{x}{y}=\frac{2H}{\sqrt{1-4H^2}}$} ;
						\node at (3.2,1) {$u_{d,+}^H$};
						\node at (-1,1.5) {$u_{d,-}^H$};
						\node at (-1,3) {$H<0$};
					\end{tikzpicture}
				\end{subfigure}
				
				\caption{Non-entire $H$-graphs invariant by hyperbolic translations (case $d=\sqrt{1-4H^2}$).}
			\end{figure}
			\item Analogously, if $d=-\sqrt{1-4H^2}$, then
			$$\lim_{s\to s_0^+}w_d(s)=\sign{H}\infty,\quad \lim_{s\to s_0^-}w_d(s)=-\sign{H}\infty,$$ and additionally $$\lim_{s\to s_0^-}u_{d,-}^H(s)=-\sign{H}\infty, \qquad \lim_{s\to s_0^+}u_{d,+}^H(s)=\sign{H}\infty.$$
			\begin{figure}[h!]
				\begin{subfigure}[b]{0.4\linewidth}
					\begin{tikzpicture}[scale=0.6]
						\draw[gray, very thin,->] (-4,0)--(-2,0) node[anchor=south] {$+\infty$}--(2,0) node[anchor=south] {$+\infty$}--(4,0) node[anchor=north] {$x$ axis};
						\draw (0,0)-- (-1.35,1.3) node[anchor=west] {$+\infty$} node[anchor=east] {$-\infty$} --(-2.7,2.6) node[anchor=south] {$\tfrac{x}{y}=-\frac{2H}{\sqrt{1-4H^2}}$} ;
						\node at (-3.1,1) {$u_{d,-}^H$};
						\node at (1,1.5) {$u_{d,+}^H$};
						\node at (1,3) {$H>0$};
					\end{tikzpicture}
				\end{subfigure}
				\begin{subfigure}[b]{0.4\linewidth}
					\begin{tikzpicture}[scale=0.6]
						\draw[gray, very thin,->] (-4,0)--(-2,0) node[anchor=south] {$-\infty$}--(2,0) node[anchor=south] {$-\infty$}--(4,0) node[anchor=north] {$x$ axis};
						\draw (0,0)-- (-1.35,1.3) node[anchor=west] {$-\infty$} node[anchor=east] {$+\infty$} --(-2.7,2.6) node[anchor=south] {$\tfrac{x}{y}=-\frac{2H}{\sqrt{1-4H^2}}$} ;
						\node at (-3.1,1) {$u_{d,-}^H$};
						\node at (1,1.5) {$u_{d,+}^H$};
						\node at (1,3) {$H<0$};
					\end{tikzpicture}
				\end{subfigure}
				
				\caption{Non-entire $H$-graphs invariant by hyperbolic translations (case $d=-\sqrt{1-4H^2}$).}
			\end{figure}
		\end{enumerate}
		These limits are motivated by the following proposition.
		\begin{proposition}\label{prop:H122}
			If $d=\pm\sqrt{1-4H^2}$ then
			\begin{equation*}
				\lim_{s\to s_0^-}u_{d,-}^H(s)=\pm\sign{H}\infty\qquad\textrm{and}\qquad\lim_{s\to s_0^+}u_{d,+}^H(s)=\mp\sign{H}\infty.
			\end{equation*}
		\end{proposition}
		\begin{proof}
			We study only the case of $u_{d,+}^H,$ for $d=-\sqrt{1-4H^2}$ and $H>0$, since the others are analogous.
			We call $\Sigma_H=(-2H e^u,\sqrt{1-4H^2}e^u,v)_{\left\{u,v\right\}\in\R}$ and denote by $\Theta_d^H$ the graph of $u_{d,+}^H$.
			
			Suppose that there exist $k\in\R,$ such that $\underset{s\to s_0^+}{\lim}u_{d,+}^H(s)=k,$ then $\Theta^H_d(\phi,\theta)$ and $\Sigma_H(\varphi,t)$ are two surfaces of constant mean curvature $H,$ with mean curvature vector pointing in the same direction, which are tangent along the curve $(\cos(\theta_0) e^\phi,\sin(\theta_0)e^\phi,k)$ (see \Cref{ContrEx}). This contradicts the Boundary Maximum Principle and concludes the proof. 
			\begin{figure}[h!]
				\centering
				\includegraphics[width=0.5\linewidth]{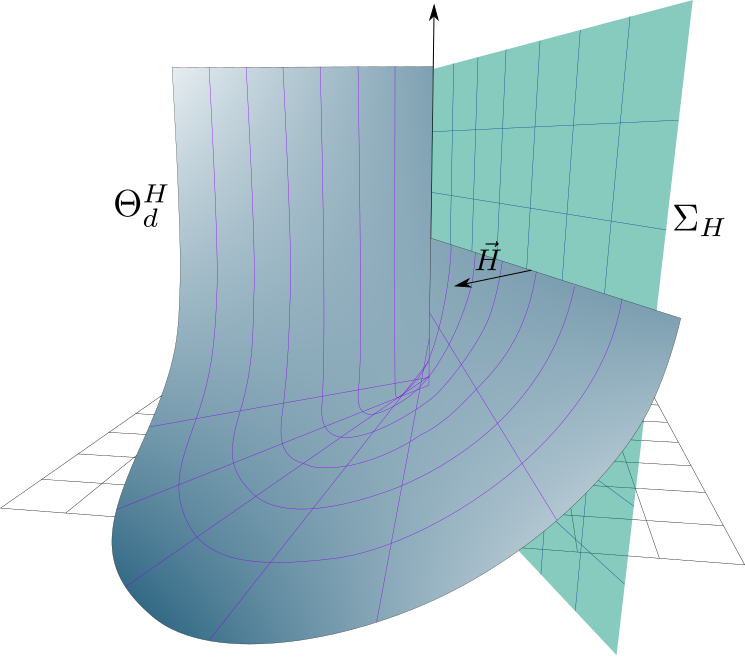}
				\caption{$\Sigma_H $ is tangent to $\Omega^H_d$ if $\underset{s\to s_0^+}{\lim}u_{d,+}^H(s)=k.$}
				\label{ContrEx}
			\end{figure}	
			
		\end{proof} 
		\item Case $d^2+4H^2>1$.
		
		In this case the polynomial $P(s)$ has two roots: $s_0^\pm=\frac{2dH\pm\sqrt{d^2+4H^2-1}}{4H^2-1}$ and then, up to vertical translation, we have again two solutions:
		$$u_{d,+}^H(s)=2\tau\arctan(s)-\sign{H}\int_{s_0^++1}^{s}\frac{\sqrt{4t^2\tau^2+t^2+}(d-2Ht)}{(t^2+1)\sqrt{1-(d-2Ht+t)(d-2Ht-t)}}dt,$$
		and
		$$u_{d,-}^H(s)=2\tau\arctan(s)-\sign{H}\int_{s}^{s_0^--1}\frac{\sqrt{4t^2\tau^2+t^2+}(d-2Ht)}{(t^2+1)\sqrt{1-(d-2Ht+t)(d-2Ht-t)}}dt.$$
		The asymptotic behaviour of $u_{d,\pm}^H(s)$ for $s\to\pm\infty$ can be studied as in \eqref{eq:asymptoticbehaviourhyperbolicsolutions}. On the other hand, it is easy to see that $w_d(s)\approx \sign{H}C^{\pm}\frac{1}{\sqrt{s-s_0^{\pm}}}$ when $s$ is near to $s_0^{\pm}$, for some constant $C^{\pm}$. In particular, $u_{d,\pm}^H$ is finite and vertical approaching $s_0^\pm$. We can notice that, up to vertical translation, the graph of $u_{d,\pm}^H\left(\tfrac{x}{y}\right)$ is the analytic extension of $u_{-d,\pm}^{-H}\left(\tfrac{x}{y}\right)$; if we glue them together, the surfaces we obtain are analogous to the tall rectangles surfaces (described for example in \cite{Cas22}) and their asymptotic boundary consists of two vertical lines in $\partial\mathbb{H}^2\times\R$.
	\end{enumerate}
	
	We now study the asymptotic behaviour of invariant $H$-surfaces. To do so, we compute the maximum of the height growth of each graph with respect to the slice $\left\{t=0\right\}$ as a function of geodesic radius of the base $\mathbb{H}^2.$
	
	The study of $\mathcal{R}^H_{d}$ is easy, since the growth function $v_{d}^H$ is already defined as a function of the geodesic radius of $\mathbb{H}^2.$ So, we only need to study its behaviour for $\rho\to\infty.$ An easy computation implies that
	\[\dot{v}^H_{d}(\rho)\approx \frac{2H\sqrt{1+4\tau^2}}{\sqrt{1-4H^2}}
	+ \frac{2 (d + 4 d \tau^2 + 8 H (-1 + 4 H^2) \tau^2)}{(1 - 4 H^2)^{\tfrac{3}{2}} \sqrt{1 + 4 \tau^2}} e^{-\rho}
	+o\left(e^{-\rho}\right),\] 
	and integrating this function we get 
	\begin{equation}\label{eq:Asymptotic-Rot}
		v^H_{d}(\rho)\approx \frac{2H\sqrt{1+4\tau^2}}{\sqrt{1-4H^2}}\rho
		-\frac{2 (d + 4 d \tau^2 + 8 H (-1 + 4 H^2) \tau^2)}{(1 - 4 H^2)^{\tfrac{3}{2}} \sqrt{1 + 4 \tau^2}}e^{-\rho}
		+o\left(e^{-\rho}\right).
	\end{equation}
	
	In a similar way we study the growth of $\Theta^H_{d}$. First we have to notice that growth function $u^H_{d}(s)$ is not written with respect to the geodesic radius of the base, so we need to use the change $s(\rho)=\sinh(\rho)$, obtaining
	\[\dot{u}^H_d(\rho)=\frac{1}{\cosh(\rho)}\left(2\tau+\frac{(-d+2H\sinh(\rho))\sqrt{1+(1+4\tau^2)\sinh^2(\rho)}}{\sqrt{1-d^2+\sinh(\rho)(4dH+\sinh(\rho)-4H^2\sinh(\rho))}}\right),\]
	that can be approximated by the function
	\[\dot{u}^H_d(\rho)\approx\pm\frac{2H\sqrt{1+4\tau^2}}{\sqrt{1-4H^2}}+\left(2\tau-2d\sqrt{\frac{1+4\tau^2}{(1-4H^2)^3}}\right)e^{\mp r}+o(e^{\mp r})\]
	when $\rho\to\pm\infty$. Integrating we obtain that 
	\begin{equation}
		\label{eq:approxu}
		u^H_d(\rho)\approx\pm\frac{2H\sqrt{1+4\tau^2}}{\sqrt{1-4H^2}}\rho\pm\left(2\tau-2d\sqrt{\frac{1+4\tau^2}{(1-4H^2)^3}}\right)e^{\mp r}+o(e^{\mp r})
	\end{equation}
	when $\rho\to\pm\infty$.
	\begin{remark}\label{rem:altregeo}
		Notice that, if we want to study the growth of $\Theta_d^H$ with respect to the radial geodesics outgoing from a fixed point $(x_0,y_0)$ with tangent vector $\left(\frac{y_0}{\cosh(\phi)},-y_0\tanh(\phi)\right)_{\phi\in\R}$, we need to use the change $$s(\rho)=\frac{x_0\cosh(\rho+\phi)+y_0\sinh(\rho)}{y_0\cosh(\phi)}.$$
		A direct computation implies that, even using this change, $u_d^H(\rho)$ is approximated by the same function \eqref{eq:approxu} for $|\rho|$ sufficiently large.
	\end{remark}
	
	The last thing to check is the asymptotic relation between $\mathcal{U}$ and $\mathcal{S}$, that will allow us to compare $\mathcal{R}^H_{d_1}$ and $\Theta^H_{d_2}.$ To do so, we use \eqref{iso:CtoH} to write $\mathcal{U}$ in the half-space model $\mathcal{H}$ and compare it to $\mathcal{S},$ that is $\left\{t=0\right\}.$ So, in the cylinder model $\mathcal{C},$ we parametrize $\mathcal{U}$, that is $\left\{t=0\right\},$ by 
	\[\mathcal{U}(x,y)=\left(\frac{-1+x^2+y^2}{x^2+(y+1)^2},\frac{-2x}{x^2+(y+1)^2},0\right),\] with $(x,y)\in\left\{\R^2\mid y>0\right\}.$ Hence, the image of $\mathcal{U}$ by the isometry \eqref{iso:CtoH} is \[(\psi^{-1}\circ \mathcal{U})(x,y)=\left(x,y,4\tau\arctan\left(\tfrac{x}{y+1}\right)\right)\] (see \cite{Cas22} for more details).
	\begin{figure}[h!]
		\centering
		\includegraphics[width=0.6\linewidth]{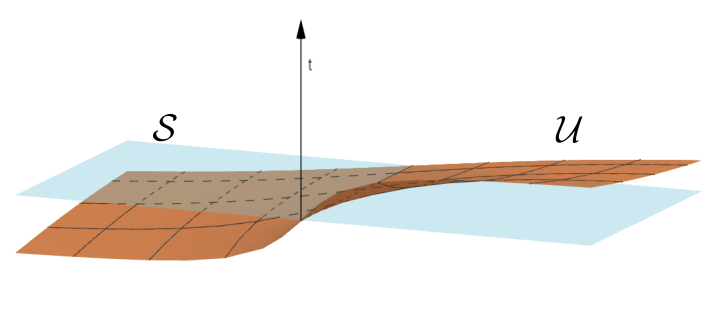}
		\caption{The umbrella surface $\mathcal{U}$ in the half-space model $\mathcal{H}$.}
		\label{UinH}
	\end{figure}	
	
	In particular we have that the vertical distance between $\mathcal{U}$ and $\mathcal{S}$ is bounded (see \Cref{UinH}). This and Remark~\ref{rem:altregeo} imply that, fixed $H\in\left(-\tfrac{1}{2},\tfrac{1}{2}\right),$ for any $d_1,d_2\in\R$ the asymptotic vertical distance between $\mathcal{R}^H_{d_1}$ and $\Theta^H_{d_2}$ is bounded by a constant.
	
	\section{A new model for $\E(-1,\tau)$}\label{sec:newModel}
	In this section we use the Killing Submersions theory to describe a new cylindrical model {$\mathcal{C}_H$ ($H$-cylindrical model)}  for $\E(-1,\tau),$ in which the slices $\left\{t=c\right\}_{c\in\R}$ are rotational surfaces of fixed constant mean curvature $-\tfrac{1}{2}\leq H\leq\tfrac{1}{2}$. 
	
	In general, if $\mathcal{D}\subset\R^2$ is a simply connected domain endowed with the conformal metric $\lambda^2(dx^2+dy^2)$, with $\lambda\in\mathcal{C}^\infty(\mathcal{D})$, and $\mathcal{D}\times\R$ is such that $\pi\colon \mathcal{D}\times\R\to \mathcal{D}$ is a unitary Killing submersion with bundle curvature $\tau\in\mathcal{C}^\infty(\mathcal{D})$, then $ \mathcal{D}\times\R$ is endowed with the metric
	\[ds^2=\lambda^2(dx^2+dy^2)+[\lambda(a(x,y)dx+b(x,y)dy)+dt]^2,\]
	where $a,b\in\mathcal{C}^\infty( \mathcal{D})$ satisfy the bundle curvature equation \[\frac{(\lambda b)_x-(\lambda a)_y}{2\lambda^2}=\tau.\] 
	For any fixed surface $\Sigma=(x,y,d(x,y))$ in $\left(\mathcal{D}\times\R,ds^2\right)$, described by the function $d\in\mathcal{C}^\infty( \mathcal{D}),$ that is an entire graph over $\mathcal{D}$, the map
	\begin{equation}\label{Iso}
		\begin{array}{rccc}
			\mathcal{I}_d\colon&\left(\mathcal{D}\times\R,ds^2\right)&\to& \left(\mathcal{D}\times\R,ds_1^2\right)\\
			&(x,y,t)&\mapsto&(x,y,t-d(x,y))
		\end{array}
	\end{equation}
	fixes the fibers of the submersion and send $\Sigma$ to the slice $\left\{t=0\right\}$ in $ \mathcal{D}\times\R$. Here, the metric $ds_1^2$ is again of the form
	\[ds_1^2=\lambda^2(dx^2+dy^2)+[\lambda(a'(x,y)dx+b'(x,y)dy)+dt]^2\]
	with $a',b'\in\mathcal{C}^\infty(\mathcal{D})$ satisfying \[\frac{(\lambda b')_x-(\lambda a')_y}{2\lambda^2}=\tau.\]
	The uniqueness result for Killing Submersions (see \cite[Theorem~2.6]{LerMan17}) implies that $\mathcal{I}_d$ is an isometry and $a$ and $ a'$ and $b$ and $b'$ satisfy the following relation:
	\[b'=b+\frac{d_y}{\lambda},\qquad a'=a+\frac{d_x}{\lambda}.\] 
	
	We now apply this idea to $\E(-1, \tau),$ described with the cylindrical model $\mathcal{C}.$ For any fixed $H\in\left[-\tfrac{1}{2},\tfrac{1}{2}\right]$, we consider $\Sigma$ to be the entire rotational $H$-graph $\mathcal{R}^H_{-2H}(\rho,\theta)$. So, we get that $\E(-1,\tau)$ is isometric to $\Omega\times\R,$ endowed with the metric
	\begin{equation}\label{RotMetric}
		ds^2=\lambda_{\mathcal{C}}(x,y)(dx^2+dy^2)+\left[\lambda_{\mathcal{C}}(x,y)\left[a_{\mathcal{C}}^H(x,y) dx +b_{\mathcal{C}}^H(x,y) dy\right]+dt\right]^2,
	\end{equation}
	where
	\begin{equation*}
		\begin{array}{lr}
			a_{\mathcal{C}}^H(x,y)=2 y \tau + 2 H x \sqrt{\frac{1 + 4 (x^2 + y^2) \tau^2}{ 
					1 - 4 H^2 (x^2 + y^2)}},  &
			b_{\mathcal{C}}^H(x,y)=-2 x \tau + 2 H y \sqrt{\frac{1 + 4 (x^2 + y^2) \tau^2}{ 
					1 - 4 H^2 (x^2 + y^2)}}.
		\end{array}
	\end{equation*}
	
	In this model the unit upwise normal vector field of the graph of a function $u\colon U\subset\Omega\to\R$ is given by
	\begin{equation*}
		N_u(x,y)=\frac{1}{\sqrt{1+\|G^Hu\|^2}}(G^Hu+\partial_t),
	\end{equation*}
	where  $\|\cdot\|$ is the norm of $\mathbb{H}^2$ in the Poincar\'e disk metric $\lambda_{\mathcal{C}}^2(dx^2+dy^2),$ and \[G^Hu= \tfrac{1}{\lambda_{\mathcal{C}}^2}\left[(u_x-\lambda_{\mathcal{C}} a_{\mathcal{C}}^H)\partial_x+(u_y-\lambda_{\mathcal{C}} b_{\mathcal{C}}^H)\partial_y\right].\]
	From this we deduce that the foliation $\left\{t=c\right\}_{c\in\R}$ is asymptotically transversal to the fibers whenever $|H|<\tfrac{1}{2}$:
	\[\lim_{x^2+y^2\to1}\Prod{N_c(x,y),\partial_t}=\lim_{x^2+y^2\to1}\frac{1}{\sqrt{1+(\lambda_{\mathcal{C}} a_{\mathcal{C}}^H)^2+(\lambda_{\mathcal{C}} b_{\mathcal{C}}^H)^2}}=\frac{1-4H^2}{2H\sqrt{1+4\tau^2}}.\]
	We also deduce that the mean curvature equation of a Killing graph described by $u$ is 
	\begin{equation}\label{RotHMeanCurvature}
		H(u)=\tfrac{1}{2}\Div\left(\frac{G^Hu}{\sqrt{1+\|G^Hu\|^2}}\right),
	\end{equation}
	where $\Div(\cdot)$ and $\|\cdot\|$ are respectively the divergence and the norm of $\mathbb{H}^2$ in the Poincar\'e disk metric. 
	\begin{remark}
		In \cite{DacLi09}, an existence and uniqueness result is established for prescribed mean curvature graphs with bounded boundary values over compact domains in any Killing submersion. This result relies on the classical Leray-Schauder theory, which is based on the availability of gradient estimates. As a consequence, the Arzelà-Ascoli Theorem yields the so-called compactness theorem, ensuring the existence of a limit for a sequence of solutions to the prescribed mean curvature equation, given uniform bounds. For more details about the existence result and the Compactness Theorem see \cite[Chapter 2 and Appendix]{DelThesis}.
		
		{Notice that a vertical graph in $\mathcal{C}_H$ is also a vertical graph in $\mathcal{C}$. In particular, if the function $u$ define the graph in $\mathcal{C}_h$, its counterpart in $\mathcal{C}$ is defined by the function $u+ \mathcal{R}^H_{-2H}$}
	\end{remark}
	
	\begin{remark}\label{rem:Hcomp}
		In this new model we can consider a product compactification, called \emph{$H$-compactification}, of $\E(-1,\tau)$
		\[\partial^H_\infty\H^2\times\R\cup\H^2\times\left\{+\infty\right\}\cup\H^2\times\left\{-\infty\right\}\] 
		such that $\partial^H_\infty\H^2\times\R$ is folieted by the asymptotic boundary of the elemets of a foliation by $H$-surfaces of $\E(-1,\tau)$ (for instance, the family $\left\{t=c\right\}_{c\in\R}$ satisfies this property).
		Notice that the isometry $\mathcal{I}\colon\E(-1,\tau)\to\E(-1,\tau)$ extends to the compactification sending the standard compactification into the $H$-compactifaction, leaving the cups $\H^2\times\left\{\pm\infty\right\}$ fixed, collapsing $\partial_\infty\H^2\times\R$ onto $\partial_\infty\H^2\times\left\{\pm\infty\right\}$ (depending on the sign of $H$) and expanding $\partial_\infty\H^2\times\left\{\mp\infty\right\}$ in $\partial_\infty\H^2\times\R$. (See Figure~\ref{Hcompactification} for an illustrated example).
	\end{remark}
	\begin{figure}[h!]
		\centering
		\includegraphics[width=0.4\linewidth]{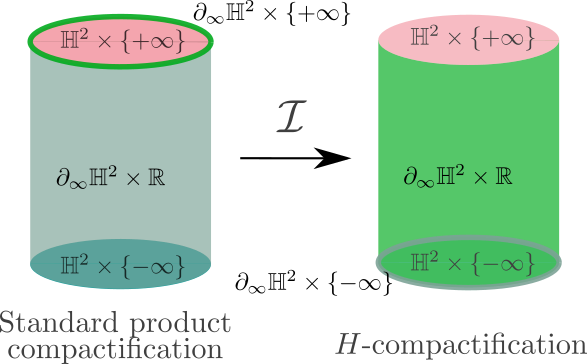}
		\caption{When $H>0$, the isometry $\mathcal{I}$ send the standard product compactification into the $H$-compactification fixing $\H^2\times\left\{+\infty\right\}$ and $\H^2\times\left\{-\infty\right\}$ and sending $\partial_\infty\H^2\times\left\{+\infty\right\}$ into $\partial_\infty\H^2\times\R$ and $\partial_\infty\H^2\times\R$ into $\partial_\infty\H^2\times\left\{-\infty\right\}$.}
		\label{Hcompactification}
	\end{figure}	
	
	\begin{definition}
		We say that a curve $\Gamma\in\partial_\infty\H^2\times\R$ in the $H$-compactification of $\E(-1,\tau)$ is in the asymptotic $H$-boundary of a surface $\Sigma$ if for any point $p\in\Gamma$ there exists a sequence $\left\{q_i\right\}$ of points of $\Sigma$ such that $\lim_{i\to\infty}q_i=p$.
		
		If $\mathcal{D}\subset\Omega$ is an unbounded domain containing in its boundary a curve $\gamma\in\partial_\infty\Omega$ and $u\in\mathcal{C}^\infty(\Omega)$, we say that has asymptotic boundary value $f\in\mathcal{C}\gamma$ if, for any sequence $\left\{p_n\right\}\subset\Omega$ converging to $p\in\gamma$, $\left\{u(p_n)\right\}$ converges to $f(p)$. 
	\end{definition}

	 {\section{Construction of the asymptotic barriers}\label{sec:Barriers}
	In this section, we describe the asymptotic $H$-barriers that will be used in the proof of the Main Theorems. Specifically, we describe the upper and lower barriers for any fixed $0<H<1/2$, noting that the construction for $0<H<-1/2$ is analogous but reversed: the description of the upper barriers for $H>0$ corresponds to that of the lower barriers for $H<0$, and vice versa.	}
	
	 {\subsection*{Lower asymptotic barriers} The main idea is to use the surfaces that are invariant by hyperbolic translation as lower asymptotic barriers. So, consider the $H$-cylindrical model $\mathcal{C}_H$ and, fixed a connected subset $\gamma^\infty\subset\partial\Omega$, let $\gamma^+$ and $\gamma^-$ the two curves with gedesic curvature equal to $2H$ connecting the endpoints of $\gamma^\infty$ and denote by $D^+$ (resp. $D^-$) the domain bounded by $\gamma^+\cup \gamma^\infty$ (resp. $\gamma^-\cup \gamma^\infty$) and $\partial \Omega$ such that $D^+$ is convex and $D^-$ is concave.
	We denote by $\omega^\pm\in\mathcal{C}^\infty(D^\pm)$ the function describing the $H$-surface invariant by hyperbolic translation that diverges to $\pm\infty$ along $\gamma^\pm$. Notice that, since all hyperbolic surfaces have the same asymptotic behaviour (see \eqref{eq:approxu}), it follows that, up to vertical translation, $\omega^+=\omega^-$ at $\gamma^\infty$. Furthermore, \Cref{rem:altregeo} and \eqref{eq:Asymptotic-Rot} implies that for any radial direction outgoing from a fixed point the rotational $H$-surfaces and the hyperbolic $H$-surface described by $\omega^-$ have the same growth. In particular, there exists a function $f\in\mathcal{C}^1(\gamma^\infty)$ such that $\omega^+=\omega^-=f$ at $\gamma^\infty$. Nevertheless, the surface described by $\omega^+$ contains an horizontal geodesic whose projection onto $D^+$ connects the endpoints of $\gamma^\infty$, which implies that the function $f$ goes to $-\infty$ approaching the vertices of $\gamma^\infty$.	
	\begin{figure}[h!]
		\begin{subfigure}[b]{0.3\linewidth}
			\centering
			\begin{tikzpicture}[scale=1.5]
				\draw[dashed] circle (1);
				\node at (0.85,0.85) {$\partial\Omega$};
				\draw (1,0) arc (60:120:2);
				\node at (-0.8,0.2) {$\gamma^-$};
				\node at (0.3,0.6) {$\omega_-$};
				\node at (0,0.15) {$-\infty$};
				\node at (0,0.85) {$f$};
			\end{tikzpicture}
		\end{subfigure}
		\begin{subfigure}[b]{0.3\linewidth}
			\centering
			\begin{tikzpicture}[scale=1.5]
				\draw[dashed] circle (1);
				\node at (0.85,0.85) {$\partial\Omega$};
				\draw (1,0) arc (-60:-120:2);
				\node at (-0.8,0.1) {$\gamma^+$};
				\node at (0.3,0.3) {$\omega_+$};
				\node at (0,-0.15) {$+\infty$};
				\node at (0,0.85) {$f$};
			\end{tikzpicture}
		\end{subfigure}
		\caption{$H$-surfaces invariant by hyperbolic translation in the $H$-compactification.}
	\end{figure}
	\begin{remark}
		The fact that the function $f$ is bounded above but not below makes the hyperbolic invariant $H$-surfaces not useful as upper barriers in the proof of the existence of entire $H$-graph with prescribed asymptotic values. Nevertheless, we can use them as upper barriers to solve the Jenkins-Serrin problem with the right hypothesis.
	\end{remark}}
	 
	 {\subsection*{Upper asymptotic barriers}
	In this subsection we want to prove the existence of upper asymptotic barriers to study the existence of entire $H$-graph with prescribed asymptotic values in the $H$-compactification of $\E(-1,\tau)$. In particular, we want to prove that for any given smooth curve $\Gamma$ in $\partial_\infty\Omega\times\R$ that projects bijectively onto $\partial_\infty\Omega$, for any $q\in\Gamma$ there exists a curve $\Gamma_1$ above $\Gamma$ in $\partial_\infty\Omega\times\R$ containing $q$ that is the asymptotic boundary of a rotational $H$-graph.
	To prove the previous statement we use the classical cylinder model $\mathcal{C}$ for $\E(-1,\tau)$, recalling that a graph $\Sigma_u$ of $u\in\mathcal{C}^\infty(\Omega)$ is asymptotic to a curve $\Gamma=f(\partial_\infty\Omega)$, for $f\in\mathcal{C}^\infty(\partial_\infty\Omega)$ in the $H$-cylinder model $\mathcal{C}_H$ if and only if the function $\bar{u}$ describing $\Sigma_u$ in the classical cylinder model $\mathcal{C}$ satisfies the following: for any $p\in\partial_\infty\Omega$ and any sequence $\left\{p_n\right\}\subset\Omega$ converging to $p$, \[\lim_{p_n\to p}\bar{u}(p_n)-\mathcal{R}_{-2H}^H(p_n)=f(p).\] 
	For the rest of this section we use the following notation:
	\begin{itemize}
		\item we denote by $u_0\in\mathcal{C}^\infty\Omega$ the rotational $H$-graph with center of rotation in $(0,0)\in\Omega$;
		\item we denote by $\mathcal{I}^\gamma_c(x,y,z)$ the hyperbolic translation with respect to the horizontal geodesic $\gamma(t)=\left\{0,-\tanh\left(\tfrac{t}{2}\right),0\right\}$;
		\item for any $c>0$, we denote by $u_c\in\mathcal{C}^\infty\Omega$ the function the describing the rotational $H$-surface $\mathcal{I}_c^\gamma(x,y,u_0(x,y))$ with center of rotation in $\gamma(c)$.
	\end{itemize}}
	
	 {For the sake of completeness, it follows from a direct computation that the hyperbolic translation with respect to the horizontal geodesic $\gamma(t)$ is given by
	\begin{eqnarray*}
			\mathcal{I}_c^\gamma(x,y,z)=&\{\frac{-2x}{-1+x^2+y^2-(1+x^2+y^2)\cosh(c)+2y\sinh(c)},
			\frac{(1+x^2+y^2)\sinh(c)-2y \cosh(c)}{-1+x^2+y^2-(1+x^2+y^2)\cosh(c)+2y\sinh(c)},  \\&	\qquad\qquad\qquad\qquad			z-4\tau \arctan\left(\frac{x}{1+y}\right)+4\tau\arctan\left(\frac{2 e^c x}{x^2+(1+y)^2-e^c(x^2+y^2-1)}\right)\}.
	\end{eqnarray*}
	Notice that \Cref{eq:Asymptotic-Rot} guaranties that $u_c$ has radial asymptotic growth with slope $\frac{2H\sqrt{1+4\tau^2}}{\sqrt{1-4H^2}}$ with respect to $\mathcal{I}_c^\gamma(x,y,0)$, for any $c\in\R$. Furthermore, a direct computation implies that for any $c\in\R$ \[\left|- \arctan\left(\frac{x}{1+y}\right)+\arctan\left(\frac{2 e^c x}{x^2+(1+y)^2-e^c(x^2+y^2-1)}\right)\right|\leq\left|\arctan\left(e^{\tfrac{-c}{2}}\right)-\arctan\left(e^{\tfrac{c}{2}}\right) \right|,\] that is the vertical distance between $(x,y,0)$ and $\mathcal{I}_c^\gamma(x,y,0)$ is bounded for any $c\in\R$. These two conditions guaranty that there exists a continuous function $f_c\in\mathcal{C}(\partial_\infty\Omega)$ such that for any point $p\in\partial_\infty\Omega$ and any sequence $\left\{p_n\right\}\subset\Omega$ converging to $p$, $\lim_{p_n\to p}u_c(p_n)-u_0(p_n)=f_c(p)$.}
	
	 {To obtain sufficient informations about $f$, we can study the difference between $u_0$ and $u_c$ along the projection of the geodesic $\gamma(t)$. In particular we have that, for $t>0$, $u_(\gamma(t)-u_c(\gamma(t))=v_{-2H}^H(t)-v_{-2H}^H(t-c)$, where $v_{-2H}^H$ is the function described in \Cref{eq:rotationalCMCgraphs}. To study $\lim_{t\to\infty}u_0(\gamma(t)-u_c(\gamma(t))$ we notice that, since $v_{-2H}^H$ is convex, and then there exists $k_t\in (0,c)$ such that \[\frac{v_{-2H}^H(t)-v_{-2H}^H(t-c)}{c}=\dot{v}_{-2H}^H(t-c+k_t).\] Hence,
	\[\lim_{t\to\infty}u_0(\gamma(t)-u_c(\gamma(t))=\lim{t\to\infty}\frac{v_{-2H}^H(t)-v_{-2H}^H(t-c)}{c} c=\lim_{t\to\infty}\dot{v}_{-2H}^H(t-c+k_t)\, c=\frac{2H\sqrt{1+4\tau^2}}{\sqrt{1-4H^2}} c\] using \Cref{eq:Asymptotic-Rot}.
	An analogous argument implies that \[\lim_{t\to-\infty}u_0(\gamma(t)-u_c(\gamma(t))=-\frac{2H\sqrt{1+4\tau^2}}{\sqrt{1-4H^2}} c.\]
	In particular, we have that for any fixed $c>0$, \[\min_{p\in\partial_\infty\Omega}f_c(p)\leq-\frac{2H\sqrt{1+4\tau^2}}{\sqrt{1-4H^2}} c\qquad\max_{p\in\partial_\infty\Omega}f_c(p)\geq\frac{2H\sqrt{1+4\tau^2}}{\sqrt{1-4H^2}} c\]
	Now fix $q\in\partial\infty\Omega$ and, for any $c>0$, let $p_c\in\partial_\infty\Omega$ the point such that $\min_{p\in\partial_\infty\Omega}f_c(p)=f_c(p_c)$ and  $Rot^q_c\colon\E(-1,\tau)\to\E(-1,\tau)$ be the rotation that fix $(x,y,u_0(x,y))$ and send $p_c$ to $q$ and denote by $u_c^q\in\mathcal{C}^\infty(\Omega)$ the function whose graph $Rot^q_c((x,y,u_c(x,y)))$ and $f_c^q\in\mathcal{C}^\infty(\partial_\infty\Omega)$ the function describing its asymptotic boundary. By construction for any $q\in\partial_\infty\Omega$ and $c>0$, we get $u^q_c-f^q_c(q)>u_0$ in $\Omega$ and \[\lim_{q_n\to q}u^q_c(q_n)-f^q_c(q)-u_0(q_n)=0\] for any sequence $\left\{q_n\right\}$ converging to $q$. Now \[\lim_{c\to\infty}u^q_c-f^q_c(q)-u_0\] diverges in all $\Omega$ since $-f^q_c(q)$ diverges to $+\infty$ for $c\to\infty$. In particular, we proved the following proposition in the $H$-cylinder model $\mathcal{C}_H$:
	\begin{proposition}\label{Prop:upperbarrier}
		Let $g\in\mathcal{C}^\infty(\partial_\infty\Omega)$ be a smooth function. Then, for any $q\in\partial_\infty\Omega$, there exists a $c$ sufficiently large such that $f_c^q> g$ in $\partial_\infty\Omega\setminus\left\{q\right\}$ and $g(q)=f^q_c(q)$ and there exists a rotationally invariant $H$-graph $\omega^{q}_c$ whose asymptotic boundary is $f^q_c(\partial_\infty\Omega)$.
	\end{proposition}}
	
	\section{Existence of entire $H$-graphs}\label{sec:EntireHGraphs}
	In this section  we consider the $H$-compactification of $\E(-1,\tau)$ in the $H$-cylinder model $\mathcal{C}_H$ given by $\Omega\times\R$ endowed with the metric in \eqref{RotMetric}, in which $\left\{t=0\right\}$ is a rotational surface of constant mean curvature $H\in\left(0,\tfrac{1}{2}\right)$ and we prove the existence of $H$-graph having a prescribed asymptotic $H$-boundary. We point out that the existence of non-invariant subcritical CMC graphs in $\mathbb{H}^2\times\R$ was already proven in \cite{FolRos22} using Scherk graphs over ideal domains. Our approach will allow us to have better control over the asymptotic behaviour of the $H$-graphs we will build.
	
	\begin{theorem}\label{thm:entire}
		Let $\Gamma$ be a rectifiable curve in $\partial\Omega\times\R$, that is the vertical graph of the function $\varphi\in\mathcal{C}(\partial\Omega)$. Then, there exists an entire $H$-graph having $\Gamma$ as asymptotic $H$-boundary. Such graph is unique.
	\end{theorem}
	\begin{proof}	
		Assume first that $\Gamma$ is differentiable. For $n\in\mathbb{N}$, let $D_n\subset\Omega $ be the disk centered at the origin with Euclidean radius $1-\tfrac{1}{n+1},$ so that $\left\{D_n\right\}_n$ is an exhaustion of $\Omega.$ Since $\partial D_n$ has geodesic curvature greater than $2H$ for any $n\in \mathbb{N},$ for any $g\in\mathcal{C}(\partial D_n)$, there exists a $H$-graph $u\in\mathcal{C}^2(D_n)$ such that $u\equiv g $ in $\partial D_n$ {(see \cite[Theorem 1]{DacLi09}, or \cite[Theorem 2.1]{DelThesis})}. Let $\eta_n\colon\partial D_n\to\partial\Omega$ be such that $\eta_n(x,y)=\left(\tfrac{n+1}{n}x,\tfrac{n+1}{n}y\right)$ and $\varphi_n=\varphi\circ\eta_n.$ Hence, $\varphi_n(\partial D_n)$ is a sequence of curves lying between $\left\{t=\min\,\varphi\right\}$ and $\left\{t=\max\,\varphi\right\}$ and converging to $\Gamma$ for $n\to\infty.$ Denote by $u_n$ the solution of the Dirichlet problem for constant mean curvature $H$ in $D_n$ with $\varphi_n$ as boundary values and by $M_n$ its graph. Recall that in this setting the slice $\left\{t=c\right\}$ is an $H$-surface for any $c\in\R$. Then, the maximum principle implies that $\min\,\varphi\leq u_n\leq\max\,\varphi$ for any $n\in\mathbb{N}.$ 
		The Compactness Theorem implies that for each $n>1$ there exists subsequence $\left\{u^n_m\right\}\subset\left\{u^{n-1}_m\right\}$ that converges to an $H$-graph $u^n$ in $D_n$, where $u^{1}_m=u_m$. By construction, for each $n\in\N$, $u^{n+1}$ is the analytic extension of $u^n$ from $D_n$ to to $D_{n+1}$. In particular, it follows by a diagonal argument that the subsequence $\left\{u^n\right\}$ of $\left\{u_n\right\}$ converges to a solution $u\colon\Omega\to\R$ uniformly on compact subsets of $\Omega$, and we denote by $M$ its graph.
		
		 {We have that $\Gamma$ is contained in the asymptotic $H$-boundary $\partial M$ of $M$ by construction, so it remain to prove that if $p\in\partial\Omega\times\R\setminus\Gamma,$ then $p\notin\partial M.$ So we want to show that for $p\notin\partial M,$ there exists a surface that separates $p$ from all the element of $\left\{u_n\right\}$. We need to distinguish two cases: $p$ above $\Gamma$ and $p$ below $\Gamma$. Let us start with the first case. Without loss of generality, we can assume that $p=(1,0,\varphi(1,0)+2\varepsilon).$ We want to use the upper barriers described in the previous section, so let $q=p=(1,0,\varphi(1,0)+\varepsilon)$ and consider as upper barrier the function $\omega^q_c$ described in \Cref{Prop:upperbarrier}. If $c$ is sufficiently large, $f^q_c>\varphi$ in all $\partial_\infty\Omega$ and $\omega_c^q> u^n$ in $D_n$. }
		
		 {If $p$ is below $\Gamma$, without loss of generality we can assume $p=(1,0,\varphi(1,0)-2\varepsilon).$ Let $\varepsilon_1>0$ be such that $\underset{q\in\partial\Omega_{\varepsilon_1}}{\max}(\varphi((1,0))-\varphi(q))<\varepsilon,$ where $\partial\Omega_{\varepsilon_1}=\left\{(x,y)\in\partial\Omega\mid |y|<\varepsilon_1,x>0\right\}.$ Denote by $\varepsilon_2=\tfrac{\varphi_1}{3}$, by $\gamma_\infty=\left\{(x,y)\in\partial\Omega\mid |y|<\varepsilon_1,x>0\right\}$ and by $\gamma$ the curve of constant geodesic curvature $2H$ connecting the endpoints of $\gamma_\infty$ such that the domain $D_\gamma$ bounded by $\gamma\cup\gamma_\infty$ and consider in $\omega^-\in\mathcal{C}^\infty(D_\gamma)$ described in the previous section whose graph is an hyperbolic invariant $H$-surface that diverges to $-\infty$ along $\gamma$ and to the function $f$ along $\gamma_\infty$. After possibly a vertical translation and a rotation we can assume that $f$ attains his maximum in $(1,0)$ and that $f(1,0)=\varphi(1,0)-\varepsilon$. In this way, the maximum principle implies that $\omega^-<u_n$ in $D_\gamma\cap D_n$ for any $n\in\N$ and so $p$ can not be contained in $\partial M$.}

		Now, if $\Gamma$ is rectifiable, we consider two families of differentiable curves approximating $\Gamma$ respectively from above and from below, and use the argument in \cite[Correction 4(a)]{NelRos07} and this conclude the proof.
	\end{proof}
	\begin{remark}\label{RemSlab}
		The result of this theorem is consistent with the Slab Theorem in \cite{HauMenRod19}. Indeed, if we apply the inverse isometry $\mathcal{I}^{-1}$ and move back to the classical cylinder model $\mathcal{C}$ with $\left\{t=0\right\}$ being rotational and minimal, we have that all the solutions given by applying \Cref{thm:entire} are asymptotic to the entire rotational $H$-graph, that is, they diverge to $+\infty$ approaching the asymptotic vertical boundary of $\mathcal{C}.$ Furthermore, the previous result shows that a half-space Theorem analogous to \cite[Theorem 1]{NS} and \cite[Theorem~3]{Maz15} cannot be proven for $-\tfrac{1}{2}<H<\tfrac{1}{2}$ and that the hypothesis of \cite[Theorem~2]{MazWan15} are sharp.
	\end{remark}
	
	 {With an analogous argument we can prove the following result.
	\begin{theorem}\label{thm:unbouded}
		Let $D\subset\Omega$ be a convex domain with $\partial D=\gamma\cup\gamma_\infty$ is such that $\gamma_\infty\subset\partial\Omega$ and $\gamma$ has geodesic curvature greater or equal to $2H$ with respect to $D$. Then, for any $f\in \mathcal{C}(\partial D)$ piecewise continuous, there exists a solution to the Dirichlet Problem
		\[\begin{cases}
			\begin{array}{ll}
				H(u)=H&\textrm{in }D,\\
				u=f&\textrm{in }\partial D.
			\end{array}
		\end{cases}\]
	\end{theorem}
	\begin{remark}
		Using the Perron Process (see \cite[Chapter 2.4]{DelThesis} for a detailed description) we extend this result and assume $D$ to be non-convex with re-entrant corners, that is, we assume $\gamma$ to be piecewise smooth with geodesic curvature greater or equal to $2H$ with respect to $D$ where it is defined.
	\end{remark}
	}
	\section{Asymptotic Jenkins--Serrin result}\label{App}
	
	Again we consider the model that describe the $H$-compactification of $\E(-1,\tau)$ given by $\Omega\times\R$ endowed with the metric in \eqref{RotMetric}, in which $\left\{t=0\right\}$ is a rotational surface of constant mean curvature $H\in\left(-\tfrac{1}{2},\tfrac{1}{2}\right).$ In this section we prove a Jenkins--Serrin type result for subcritical constant mean curvature graphs over unbounded domains of $\mathbb{H}^2.$ When $\tau=0$ this result was proven by Mazet, Rodriguez and Rosenberg (see \cite[Theorems~4.9,~4.12]{MazRodRos11}) for $H=0$ and by Folha and Melo (see \cite[Theorems~3.1,~3.2]{FolMe12}) for $0<H<\tfrac{1}{2}$ when the asymptotic boundary of the domain consists only of a finite number of isolated points of the asymptotic boundary of $\mathbb{H}^2.$ We extend these results to the case $\tau\neq0 $ allowing the asymptotic boundary of the domains to contain open arcs of the asymptotic boundary of $\mathbb{H}^2.$ 
	
	\begin{definition}[Jenkins--Serrin domain] \label{def:admissible domain}
		A simply connected domain $\mathcal{D}\subset\mathbb{H}^2$ is called a \emph{Jenkins--Serrin domain} (see \Cref{AdmissibleDomain}) if its boundary $\partial_\infty\mathcal{D}$ consists of a finite number of arcs $A_i$ and $B_i$ of constant geodesic curvature $2H$ and $-2H$ respectively with respect to $\mathcal{D}$, a finite number of arcs $C_i$ of geodesic curvature $k_g(C_i)\geq 2|H|$ and a finite number of open arcs $D_i\subset\partial_\infty\mathbb{H}^2,$ together with their endpoints, which are called the vertices of $\mathcal{D}.$
		
		If $H>0$ (resp. $H<0$), we say that a Jenkins--Serrin domain $\mathcal{D} $ is \emph{admissible} if 
		\begin{enumerate}
			\item[A)] neither two $A_i$ or two $B_i$ meets at a convex corner of $\mathcal{D},$ nor at an asymptotic vertex;
			\item[B)] denoting by $B_i^*$ (resp. $A_i^*$) the reflection of $B_i$ (resp. $A_i$) in $\mathbb{H}^2$ with respect to the geodesic passing through the endpoints of $B_i$ (resp. $A_i$), for any $i\in I,$ the interior of $B_i^*$ (resp. $A_i^*$) does not intersect $\overline{\mathcal{D}};$
			 {\item[C)] if $\left\{D_i\right\}\neq\emptyset$ and $H>0$ (resp. $H<0$),for any $i,j$, $D_j$ and $A_i$ (resp. $B_i$) have no common endpoints; 
			\item[D)]if $\left\{D_i\right\}\neq\emptyset$ and $H>0$ (resp. $H<0$), the curve $\gamma_i$ of constant geodesic curvature $2H$ such that the domain bounded by $\gamma_i\cup D_i$ is concave (resp. convex) is contained in $\mathcal{D}$.}
		\end{enumerate}	
		
		If $\mathcal{D}$ is an admissible Jenkins--Serrin domain, we call \emph{extended Jenkins--Serrin domain} the domain $\mathcal{D}^*$ whose boundary $\partial_\infty\mathcal{D}^*$ consists of the union of the arcs $A_i$, $B_i^*$, $C_i$ and $D_i$ (resp. $A_i^*$, $B_i$, $C_i$ and $D_i$).
		
	\end{definition}
	
	\begin{remark}
		Condition (A) of admissibility is a necessary condition that follows by applying the Flux Argument. Condition (B), is necessary to define the extended Jenkins--Serrin domain, where is possible to solve a Dirichlet problem with finite boundary data. { Condition (C) and (D) allows construction of the right upper/lower barriers, depending on the sign of $H$ (see \Cref{sec:Barriers}).
		Notice also that condition (D) is trivially satisfied when $\mathcal{D}$ does not admit re-entrant corners.}
	\end{remark}
	
	Without loss of generality we can assume $H>0$; the argument for $H<0$ is the same under minor changes on the definition of the admissible domain (see \Cref{def:admissible domain}) and on the choice of the upper and lower barriers (see \Cref{sec:Barriers}). We aim to give necessary and sufficient conditions such that the following Dirichlet problem admits a solution.
	
	\begin{definition}[Jenkins--Serrin Problem]
		Fin $H\in(0,1/2)$ and consider the $H$-cylinder model for $\E(-1,\tau)$.
		Let $\mathcal{D}\subset\Omega$ be an admissible Jenkins--Serrin domain, $f_i\in\mathcal{C}(C_i)$ and $g_i\in\mathcal{C}(D_i)$ be such that their images are rectifiable curves and $g_i$ are bounded above. Then, we call Jenkins--Serrin Problem the Dirichlet Problem \[\mathcal{P}_{JS}(\mathcal{D}):=
		\begin{cases}
			\begin{array}{ll}
				H(u)=H&\textrm{in }\mathcal{D},\\
				u=+\infty&\textrm{in }\cup A_i,\\
				u=-\infty&\textrm{in }\cup B_i,\\
				u=f_i&\textrm{in } C_i,\\
				u=g_i&\textrm{in } D_i.
			\end{array}
		\end{cases}\]
	\end{definition} 
	Before stating the theorem, in order to adapt the classical argument we need to define what an inscribed $H$-polygon is.
	\begin{definition}[Inscribed $H$-polygon]
		Let $\mathcal{D}$ be an admissible domain. We say that $P$ is an admissible inscribed $H$-polygon if $P \subset \mathcal{D}\cup\partial_\infty\mathcal{D},$ its edges have constant geodesic curvature $\pm2H$ and all the vertices of $P$ are vertices of $\mathcal{D}.$
	\end{definition}
	For each ideal vertex $p_i$ of $\mathcal{D}$ at $\partial_\infty\mathbb{H}^2,$ we consider a horocycle $H_i$ at $p_i.$	Assume $H_i$ is small enough so that it does not intersect bounded edges of $\partial_\infty\mathcal{D}$ and $H_i\cap H_j = \emptyset$ for every $i \neq j.$ 
	Let $F_i$ be the convex horodisk with boundary $H_i.$ Each $A_i$ meets exactly two
	horodisks. Denote by $\tilde{A}_i$ the compact arc of $A_i$ which is the part of $A_i$ outside the	two horodisks; we denote by $|A_i|$ (resp. $|B_i|$) the length of $\tilde{A}_i$ (resp. $\tilde{B}_i$).
	Let $P\subset\mathcal{D}$ be an inscribed $H$-polygon and denote by $\mathcal{D}_P\subset\mathcal{D}$ the domain bounded by $P.$ For each arc $\eta_j\in P,$ we define $\tilde{\eta}_j$ and $|\eta_j|$ in the same way.
	\[\gamma(P)=\sum_{j\in J}|\eta_j|,\qquad\alpha(P)=\sum_{A_i\subset P}| A_i|\quad\textrm{and}\quad\beta(P)=\sum_{B_i\in P}|B_i|.\]
	For any family of horocycles $\mathcal{H}=\left\{H_i\right\},$ denote by $F_i$ the convex horodisk bounded by $H_i,$ $\tilde{\mathcal{D}}_P=\bigcup_{i\in I}(\mathcal{D}_P\cap F_i)$, {where $i\in I$ if and only if the area $\mathcal{A}(\mathcal{D}_P\cap F_i)$ of $\mathcal{D}_P\cap F_i$ is finite} and $\mathcal{D}_P^\mathcal{H}=\mathcal{D}_P\setminus\tilde{\mathcal{D}}_P$ and we define \[\tilde{A}(\mathcal{D}_P):=\mathcal{A}(\mathcal{D}_P^\mathcal{H})+\mathcal{A}(\tilde{\mathcal{D}}_P).\]
	Notice that this definition makes sense, since $\tilde{\mathcal{A}}(\mathcal{D}_P)$ is finite (see \cite[Section~3]{FolMe12} for more details on the choice of the horodisks).
	
	\begin{figure}[h!]
		\centering
		\begin{tikzpicture}[scale=3]
			\draw[dashed] circle (1);
			\draw (1,0) arc (0:30:1);
			\draw (1,0) arc (120:180:1);
			\draw (0.866,1/2) arc (-60:-120:1);
			\draw (0.866-1,1/2) arc (85:118:0.5);
			\draw (-1/2,-0.866) arc (240:300:1);
			\draw (-1/2,-0.866) arc (-45:37:1);
			\node at (0.3,-0.2) {$\mathcal{D}$};
			\node at (0.35,0.45) {$C_1$};
			\node at (1.1,0.2) {$D_1$};
			\node at (0,-1.1) {$D_2$};
			\node at (-0.3,-0.2) {$B_1$};
			\node at (-0.35,0.55) {$A_1$};
			\node at (0.7,-0.4) {$B_2$};
		\end{tikzpicture}
		\caption{An example of admissible Jenkins-Serrin domain.}
		\label{AdmissibleDomain}
	\end{figure}
	\begin{theorem}\label{thm:JS}
		Let $\mathcal{D}$ be an admissible Jenkins--Serrin domain. Assume that both $\left\{C_i\right\}$ and $\left\{D_i\right\}$ are empty. Then, there exists a solution to the Dirichlet problem $\mathcal{P}_{JS}(\mathcal{D})$ if and only if for some choice of horocycles at the vertices we have
		\begin{equation}\label{eq:JScond0}
			\alpha(\partial\mathcal{D})=\beta(\partial\mathcal{D})+2H\tilde{\mathcal{A}}(\mathcal{D}),
		\end{equation}
		and for any inscribed $H$-polygon $P,$
		\begin{equation}\label{eq:JScond}
			2\alpha(P)<\gamma(P)+2H\tilde{\mathcal{A}}(\mathcal{D}_P)\qquad\textrm{and}\qquad2\beta(P)>\gamma(P)-2H\tilde{\mathcal{A}}(\mathcal{D}_P).
		\end{equation}
		
		If $\mathcal{D}$ is relatively compact, the solution is unique up to vertical translation.
		
		If either $\left\{C_i\right\}$ or $\left\{D_i\right\}$ is not empty, there exists a solution to $\mathcal{P}_{JS}(\mathcal{D})$ if and only if for some choice of horocycles at the vertices, the inequalities in \eqref{eq:JScond} are satisfied for any $H$-polygon $P$ inscribed in $\mathcal{D}.$
		If $\mathcal{D}$ is relatively compact, the solution is unique.
	\end{theorem}
	\begin{remark}
		Notice that conditions \eqref{eq:JScond0} and \eqref{eq:JScond} make sense, that is, they are preserved for smaller horocycles (for a proof of this property of the hyperbolic space see \cite[Section 7]{FolMe12}).
	\end{remark}
	
	\subsection*{General properties of graphs of constant mean curvature}
	Before, introducing the problem we recall a couple of results that will be necessary to prove the theorem.
	Let $D\subset\Omega$ be an open domain and $u\in\mathcal{C}^2(D)\cap\mathcal{C}^0(\partial D)$ defining an $H$-graph  $\Sigma_u$ in the classical cylinder model $\mathcal{C}$, that is, $u$ satisfy $H(u)=H\in\R,$ then, integrating $\Div\left(\frac{Gu}{\sqrt{1+\| Gu\|^2}}\right)$ gives
	\begin{equation}\label{eq:flux}
		2H\mathcal{A}(D)=\int_{\partial D}\Prod{\frac{Gu}{\sqrt{1+\|Gu\|^2}},\nu},
	\end{equation}
	where $\mathcal{A}(D)$ is the hyperbolic area of $D$ and $\nu$ is the outer normal to $\partial D.$ The integral in the right side of the equation is called \emph{flux} of $u$ across $\partial D.$ If $\gamma\subset\partial D$ is diffeomorphic to a segment we can define the flux of $u$ across $\gamma$ as follows:
	\begin{definition}\label{def:Flux}
		Let $\Gamma\subset D$ any smooth curve such that $\gamma\cup\Gamma$ bound a simply connected domain $\Delta_\Gamma.$  Then we define the flux of $u$ across $\gamma$ as
		\begin{equation*}			
			\Flux{u}{\gamma}=2H\mathcal{A}(\Delta_\Gamma)-\int_{\Gamma}\Prod{\frac{Gu}{\sqrt{1+\|Gu\|^2}},\nu}
		\end{equation*}
	\end{definition}
	Notice that this definition does not depend on the choice of $\Gamma.$ 
	
	We now state two Lemmas that we need to construct the barriers and to prove  \Cref{thm:JS}. The proof is analogous to the case $\tau=0,$ proved in \cite[Section~5]{HauRoSpr09}.
	
	\begin{lemma}\label{lemma:flux0}
		Let $u$ be such that $H(u)=H$ in a bounded domain $D$, let $\Gamma$ be a piecewise $\mathcal{C}^1$ curve in $\partial D$ and denote by $|\Gamma|$ the length of $\Gamma$. Then:
		\[2H \mathcal{A}(D)=\int_{\partial D}\Prod{\frac{Gu}{\sqrt{1+\|Gu\|^2}},\nu}\,\textrm{and }\left|\int_{\Gamma}\Prod{\frac{Gu}{\sqrt{1+\|Gu\|^2}},\nu}\right|\leq|\Gamma|.\]
		Furthermore, if $u$ is continuous on $\Gamma$, then
		\[\left|\int_{\Gamma}\Prod{\frac{Gu}{\sqrt{1+\|Gu\|^2}},\nu}\right|<|\Gamma|.\]
		In $u$ is bounded in $D$ and diverges to $\pm\infty$ on $\Gamma$, then
		\[\int_{\Gamma}\Prod{\frac{Gu}{\sqrt{1+\|Gu\|^2}},\nu}=\pm|\Gamma|.\]
	\end{lemma}
	
	\begin{lemma}\label{lemma:flux3}
		Let $D$ be a domain bounded in part by an arc $\gamma$ and let $\left\{u_n\right\}$ be a sequence of solutions of $H(u_n)=H$ in $D,$ such that each $\left\{u_n\right\}$ is continuous on $\gamma.$ Then
		\begin{enumerate}
			\item if the sequence tends to $+\infty$ uniformly on compact subsets of $\gamma$ while remaining uniformly bounded on compact subsets of $D,$ we have
			\[\underset{n\to\infty}{\lim} \int_{\gamma}\Prod{\frac{Gu_n}{\sqrt{1+\|Gu_n\|^2}},\nu}=|\gamma|;\]
			\item if the sequence tends to $-\infty$ uniformly on compact subsets of $\gamma$ while remaining uniformly bounded on compact subsets of $D,$ we have
			\[\underset{n\to\infty}{\lim} \int_{\gamma}\Prod{\frac{Gu_n}{\sqrt{1+\|Gu_n\|^2}},\nu}=-|\gamma|.\]
		\end{enumerate}
	\end{lemma}
	
	In the following theorem we prove a General Maximum Principle for unbounded domains that will be useful to study the case $\left\{D_i\right\}=\emptyset$. For $\tau=0$, this theorem was proved by Collin and Rosenberge \cite[Theorem 2]{CoRo}, when $H=0$, and by Folha and Melo \cite[Theorem 4.2]{FolMe12}, when $H\in\left(0,\tfrac{1}{2}\right)$. The following results can be deduced by the argument in the proof of \cite[Theorem 4.1]{Del}, but we give a proof for the sake of completeness.
	\begin{theorem}
		Let $\Omega_0$ be a domain such that $\partial \Omega_0$ is an ideal $H$-polygon. Let $W\subset\Omega_0$ be a domain and let $u_1,u_2\in\mathcal{C}^0(\overline{W})$ be two solution of $H(u^i)=H$ in $W$ with $u_1<u_2$ on $\partial W$. Suppose that for each vertex $p$ of $\partial\Omega_0$, $\liminf \dist_{\H^2}(\Gamma_1,\Gamma_2)\to0$ as one converges to $p$, where $\Gamma_1$ and $\Gamma_2$ are the curves on $\partial\Omega_0$ with $p$ as vertex. Then $u_1\leq u_2$ in $W$. 
	\end{theorem}
	\begin{proof}
		Suppose that there exists $W_0\subset W=\left\{p\in W|u_1(p)>u_2(p)\right\}$. The classical Maximum Principle implies that $W_0$ cannot be compact. Furthermore, $u_1=u_2$ on $\partial W_0$ by continuity. Denote by $W_0^n$ the intersection between $W_0$ and the geodesic ball $B(p,n)$ of radius $n$ centered at a fixed point $p$ outside $W_0$. Since $H(u_1)=H(u_2)=H$, \Cref{lemma:flux0} implies that 
		\[\int_{\partial W_0^n}\Prod{\frac{Gu_1}{\sqrt{1-\|Gu_1\|^2}}-\frac{Gu_2}{\sqrt{1-\|Gu_2\|^2}},\nu}=0\]
		for any $n\in\N$, where $\nu$ is the unit normal of $\partial W_0$ pointing inside $W_0$.
		Furthermore, \Cref{lemma:flux0} implies that 
		\[\left|\int_{\partial B(p,n)\cap W_0}\Prod{\frac{Gu_1}{\sqrt{1-\|Gu_1\|^2}}-\frac{Gu_2}{\sqrt{1-\|Gu_2\|^2}},\nu}\right|\leq 2\Length(\partial B(p,n)\cap W_0),\]
		that goes to $0$ for $n\to\infty$ by hypothesis. So, the result follows by proving that $$\int_{B(p,n)\cap\partial W_0}\Prod{\frac{Gu_1}{\sqrt{1-\|Gu_1\|^2}}-\frac{Gu_2}{\sqrt{1-\|Gu_2\|^2}},\nu}$$ stays bounded away from zero. In particular, it is sufficient to prove that $\Prod{\frac{Gu_1}{\sqrt{1-\|Gu_1\|^2}}-\frac{Gu_2}{\sqrt{1-\|Gu_2\|^2}},\nu}$ does not change sign.
		
		Since, $u_1-u_2>0$ in $W_0$ and $u_1=u_2$ on $\partial W_0$, one has $Gu_1-Gu_2=\nabla (u_1-u_2)=\lambda \eta$, where $\eta$ (the inner pointing conormal to $\partial W_0$ in $W_0$) orients $\partial W_0$ and $\lambda$ is a strictly positive function (by the boundary maximum priciple). So, $\Prod{\frac{Gu_1}{\sqrt{1-\|Gu_1\|^2}}-\frac{Gu_2}{\sqrt{1-\|Gu_2\|^2}},\nu}$ does not change sign if and only if $\Prod{\frac{Gu_1}{\sqrt{1-\|Gu_1\|^2}}-\frac{Gu_2}{\sqrt{1-\|Gu_2\|^2}},Gu_1-Gu_2}$ does not change sign, but in \cite[Lemma 2.1]{Del} we can see that \[\Prod{\frac{Gu_1}{\sqrt{1-\|Gu_1\|^2}}-\frac{Gu_2}{\sqrt{1-\|Gu_2\|^2}},Gu_1-Gu_2}\geq0\] and that the equality holds only at points $q\in\partial W_0$ where $\nabla (u_1(q)-u_2(q))=0$ and this concludes the proof.
	\end{proof}
	{
	\begin{remark}
		Notice that if $u\in\mathcal{C}^\infty(D)$ defines the graph $\Sigma_u$ in $\mathcal{C}$, the quantity $\frac{Gu}{\sqrt{1+\|Gu\|^2}}$ is the horizontal part of the upwise unite normal $N_u$ of $\Sigma_u$ and does not depend on the model we are considering. In particular, the function $v=u-\mathcal{R}^H_{-2H}$ will define the same surface in the model $\mathcal{C}_H$ and \[\frac{Gu}{\sqrt{1+\|Gu\|^2}}=\frac{G^Hv}{\sqrt{1+\|G^Hv\|^2}}.\] 
		In particular, the previous results are also true for CMC graphs in the $H$-cylinder model  by defining 
		\[\Flux{u}{\gamma}=2H(u)\mathcal{A}(\Delta_\Gamma)-\int_{\Gamma}\Prod{\frac{G^Hu}{\sqrt{1+\|G^Hu\|^2}},\nu},\]
		where $H(u)$ is defined in \Cref{RotHMeanCurvature}.
	\end{remark}
	}
	
	Finally, the following theorem, proved in \cite[Theorem~3.3]{RoSoTo10}, justify the $A_i,B_i$:
	
	\begin{theorem}\label{thm:boundaryform}
		Let $D\subset\mathbb{H}^2$ be a domain bounded in part by a $\mathcal{C}^2$ arc $\gamma$ and let $u\in\mathcal{C}^2(D)$ be a solution of $H(u)=0.$ If $u$ tend to $+\infty$ for any approach to interior points of $\gamma,$ then $\gamma$ has geodesic curvature $\kappa_g(\gamma)=2H,$ while if $u$ tends to $-\infty$ for any approach to interior points of $\gamma,$ then $\gamma$ has geodesic curvature $\kappa_g(\gamma)=-2H.$
	\end{theorem}
	 {\subsection*{Divergence lines}
	The proof of \Cref{thm:JS} relies the study the divergence set of certain sequence of solution and on the application of the Divergence lines technique.
	The Divergence Lines Technique was introduced by Mazet, Rodriguez and Rosenberg to study the minimal case in $\H^2\times\R$ \cite{MazRodRos11} and extended to the CMC case by Folha and Melo \cite{FolMe12}, we refer to \cite[Section 6]{FolMe12} for more details.
	\begin{definition}
		Lef $\mathcal{D}\subset\Omega$ be a domain with piecewise smooth boundary, and $u_n$ a sequence of $H$-graphs in $\mathcal{D}$. We define the \emph{convergence set} as \[\mathcal{U}=\left\{p\in\mathcal{D}|\left\{\|\nabla u_n(p)\|\right\}\textrm{ is bounded independent of }n\right\}\]
		and the diverngence set as \[\mathcal{V}=\mathcal{D}-\mathcal{U}.\]
	\end{definition}
	A classical argument (see for instance \cite[Lemmas 6.1 and 6.2]{FolMe12}) guarantees that for any $p\in\mathcal{U}$ the sequence $\left\{u_n-u_n(p)\right\}$ admits a subsequence that converges uniformly in a neighborhood of $p$, while if $p\in\mathcal{V}$, up to pass to a subsequence, $\left\{N_{u_n}(p)\right\}$ converges to an horizontal vector field and the sequence of the graph $\left\{\Sigma_n\right\}$ of $\left\{u_n-u_n(p)\right\}$ converges to the vertical cylinder above $L$ where $L\subset\mathcal{D}$ is a curve passing through $p$ with constant curvature $2H$ such that its unit normal in $p$ coincide with the limit of $\left\{N_{u_n}(p)\right\}$.
	So we can give the following definition.
	\begin{definition}
		Given a sequence $\left\{u_n\right\}$ of function satisfying the CMC equation equal to $H$ in $\mathcal{D}$, we say that a curve $L\subset\mathcal{D}$ of geodesic curvature $2H$ is a \emph{divergence line} for $\left\{u_n\right\}$ if for any $p\in L$, up to take a subsequence, the sequence of the $H$-graphs $\left\{\Sigma_n\right\}$ of $\left\{u_n-u_n(p)\right\}$ converges to the vertical cylinder above $L$.
	\end{definition}
	It can be prove that $\mathcal{V}$ is the union of a countable number of pairwise disjoint divergent lines and that if $\mathcal{U}'\subset\mathcal{U}$ is a connected component then the limit of $\left\{u_n-u_n(p)\right\}$ diverges to $\pm\infty$ approaching $L\subset\overline{\mathcal{U}}'\cap\overline{\mathcal{V}}$ depending on the sign of the geodesic curvature of $L$ with respect to $\mathcal{U}'$ ($+\infty$ if it $L$ is convex and $-\infty$ otherwise).
	Furthermore, as it is done in \cite[Section 6]{FolMe12}, it can be shown that a divergence line can not end in the interior any $A_i,$ $B_i$ or $C_i$.}
	
	\subsection*{Proof of the Jenkins--Serrin type Theorem}

	Now we have all the ingredient to prove \Cref{thm:JS}.
	\begin{proof}[Proof of \Cref{thm:JS}]		
		The proof is divided in three steps:
		\begin{enumerate}
			\item First, one proves the result assuming that the domain is bounded.
			\item Then, one allows $\mathcal{D}$ to be unbounded by assumes $\left\{D_i\right\}$ to be empty.
			\item Finally, we study the case $\left\{D_i\right\}\neq \emptyset$.
		\end{enumerate}
		
		Recall that for $\tau=0$ Step (1) has already been proven in \cite{HauRoSpr09} and Step (2) has already been proven in \cite{FolMe12}. While the proof in this general case is identical to the case $\tau = 0$, for the sake of completeness, we will revisit the main arguments, referring to the original results for further details.
		
		Step (1). The uniqueness can be proved as in \cite[Theorem 7.13]{HauRoSpr09}, using \cite[Lemma 2.1]{Del} instead of \cite[Lemma 2.1]{HauRoSpr09}. The proof of this step is divided in four cases.
		\begin{itemize}
			\item Case $\left\{B_i\right\}=\emptyset$ and $f_i$ bounded below (\cite[Theorem 7.9]{HauRoSpr09}). 
			
			Let $u_n$ be the unique solution of the Dirichlet problem
			\[\begin{cases}
				\begin{array}{lr}
					H(u_n)=H&\textrm{in }\mathcal{D};\\
					u_n=\min(n,f_i)&\textrm{in }C_i;\\
					u_n=n&\textrm{in }A_i,						
				\end{array}
			\end{cases}\]
			so that $\left\{u_n\right\}$ be an increasing sequence of $H$-graphs.
			For each $C_i$ and any two points $p_1,p_2\in C_i$, let $\Gamma_i$ be the curve of constant geodesic curvature $2H$ passing through $p_1$ and $p_2$ such that the convex unbounded domain $\Omega_{C_i}$ bounded by $\Gamma_i(p_1,p_2)$ intersects the connected component $C_i(p_1,p_2)$ between $p_1$ and $p_2$ of $C_i$. Denoting by $\omega^+\in\mathcal{C}^\infty(\Omega_{C_i})$ the function whose graph is an hyperbolic invariant $H$-surface that diverges to $+\infty$ along $\Gamma_i(p_1,p_2)$ and converges to the graph of the function $f$ on $\partial_\infty\Omega_{C_i}$ that was described in \Cref{sec:Barriers}. Un to apply a vertical translation, since both $f_i$ and $\omega^+$ are bounded in $C_i(p_1,p_2)$, we can assume $\omega^+_{|C_i(p_1,p_2)}>f_i$ and so the Maximum Principle implies that $\omega^+>u_n$ in $\mathcal{D}\cap\Omega_{C_i}$ for any $n\in\N$. Since this is true for all $p_1,p_2\in C_i$, if the divergence set $\mathcal{V}$ of $\left\{u_n\right\}$ is not empty, it is contained in the domain $\mathcal{D}_0$ bounded by $\cup A_i\cup\Gamma_i$ (where $\Gamma_i$ is the curve of constant geodesic curvature $2H$ passing trough the endpoints of $C_i$ that is convex with respect to $C_i$). Let $\mathcal{P}\subset\overline{\mathcal{D}_0}$ be an admissible $H$-polygon bounding a connected component $\mathcal{V}_0$ of $\mathcal{V}$. It is clear that the elements of $\mathcal{P}\setminus\cup A_i$ are concave with respect to $\mathcal{D}_0$. Then, \Cref{lemma:flux0} applied to each $u_n$ in $\mathcal{P}$,
			\begin{equation}\label{eq:proof1}
				2H\mathcal{A}(\mathcal{D}_\mathcal{P})=\Flux{u_n}{\mathcal{P}\setminus\cup A_i}+\Flux{u_n}{\bigcup(A_i\cap\mathcal{P})}.
			\end{equation}
			Then, \Cref{lemma:flux3} implies that
			\begin{equation*}
				\lim_{n\to\infty}\Flux{u_n}{\mathcal{P}\setminus\cup A_i}=-(\gamma(\mathcal{P})-\alpha(\mathcal{P})).
			\end{equation*}
			But, $|\Flux{u_n}{\bigcup(A_i\cap\mathcal{P})}|\leq\alpha(\mathcal{P})$, hence $\gamma(\mathcal{P})-\alpha(\mathcal{P})\leq\alpha(\mathcal{P})-2H\mathcal{A}(\mathcal{D}_\mathcal{P})$, contradicting \eqref{eq:JScond}. In particular, $\mathcal{V}=\emptyset$ and $\left\{u_n\right\}$ converges to $u$ uniformly in all compact subsets of $\mathcal{D}$ and a standard barrier argument implies that $u$ satisfies the boundary conditions in of $P_{JS}(\mathcal{D})$.  {For completeness we show that for an arbitrary $p\in C_i$, $u(p)=f_i(p)$. To do so it is sufficient to find a neighbourhood $\Lambda\subset\overline{\mathcal{D}}$ of $p$ and a barrier $\omega_p\in\mathcal{C}^\infty(\Lambda)$ such that $\omega_p(p)=f_i(p)$, $\omega_i\geq f_i$ in $C_i\cap \Lambda$ and $\omega_i>u_n$ in $\partial\Lambda\cap\mathcal{D}$ for every $n$. So, we consider $\Lambda$ to be the intersection of $\overline{\mathcal{D}}$ and a geodesic ball centered in $p$ of radius $\varepsilon>0$ sufficiently small and in $\Lambda$ we consider $H$-graph $\omega_p$ that is equal to $f_i$ in $C_i\cap\partial\Lambda$ and $\omega^+_{|C_i(p_1,p_2)}$ (assuming $p$ between $p_1,p_2\in C_i\cap\partial\Lambda$). This solution exists for \cite[Therom 1]{DacLi09} and satisfy the assumption for the Maximum Principle. We also should point out that we can use $u_0$ as lower barrier.} The necessity of condition \eqref{eq:JScond} is clear by using \eqref{eq:proof1} and \Cref{lemma:flux0}.
			
			\item Case $\left\{A_i\right\}=\emptyset$ and $f_i$ bounded above (\cite[Theorem 7.10]{HauRoSpr09}).
			Let $u_n$ be the unique solution of the Dirichlet problem
			\[\begin{cases}
				\begin{array}{lr}
					H(u_n)=H&\textrm{in }\mathcal{D}^*;\\
					u_n=\max(-n,f_i)&\textrm{in }C_i;\\
					u_n=-n&\textrm{in }B_i^*,						
				\end{array}
			\end{cases},\]
			where $\mathcal{D}^*$ is the extended domain described in \Cref{def:admissible domain} and $B_i^*$ is the hyperbolic reflection of $B_i$ with respect to the geodesic passing through its endpoints. In this way $\left\{u_n\right\}$ is a decreasing sequence of solution to the CMC equation. Notice first that $\left\{u_n\right\}$ diverges to $-\infty$ in $\mathcal{D}^*\setminus\mathcal{D}$. Indeed the previous step implies that, for each $i$, there exists a solution $\omega_i$ of the CMC equation in the domain $\mathcal{D}_{B_i}^*$ bounded by $B_i$ and $B_i^*$ such that $\omega_i=0$ in $B_i^*$ and $\omega_n^i=+\infty$ in $B_i$. In particular, $\omega_i-n>u_n$ in $\mathcal{D}_{B_i}^*$ for any $n\in\N$ and so, taking the limit for $n\to\infty$ we get that $\left\{u_n\right\}$ diverges uniformly on $\mathcal{D}^*\setminus\mathcal{D}$. 
			 {If there is a $C_i$ with geodesic curvature not globally equal to $2H$, we can use a similar argument as in the previous case (considering a lower barrier invariant by hyperbolic translation) to show that $\left\{u_n\right\}$ converges in a neighbourhood of that $C_i$ and hence in all $\mathcal{D}$.
			Otherwise, it is sufficient to show that fixed any $p\in\mathcal{D}$ the sequence $\left\{u_n(p)\right\}$ does not diverges. Assume by contradiction that it diverges, then the sequence $\left\{u_n-u_n(p)\right\}$ converges to $u$ in a neighbourhood of $p$ that can be extended to all the interior of $\mathcal{D}$, notice that since $u_n(p)$ diverges to $-\infty$, then $u$ must diverge to $+\infty$ along all the $C_i$, but applying the flux argument this contradicts the hypothesis of the theorem. So $u_n$ converges uniformly to $u$ in the interior of $\mathcal{D}$ and we can use the same argument as before to prove that it attains the right boundary values. }
			
			\item Case $\left\{C_i\right\}\neq\emptyset$ (\cite[Theorem 7.11]{HauRoSpr09}).
			
			In this case we consider the decreasing sequence $\left\{u_n\right\}$ of solution of 
			\[\begin{cases}
				\begin{array}{lr}
					H(u_n)=H&\textrm{in }\mathcal{D}^*;\\
					u_n=f_i &\textrm{in }C_i;\\
					u_n=+\infty&\textrm{in }A_i,	\\
					u_n=-n&\textrm{in }B_i^*,	
				\end{array}
			\end{cases}\] 
			and the increasing sequence $\left\{v_n\right\}$ of solution of 
			\[\begin{cases}
				\begin{array}{lr}
					H(v_n)=H&\textrm{in }\mathcal{D};\\
					v_n=f_i &\textrm{in }C_i;\\
					v_n=n&\textrm{in }A_i,	\\
					v_n=-\infty&\textrm{in }B_i,	
				\end{array}
			\end{cases}\]
			whose existence was proved in the previous two steps. Notice that for any $m,n\in\N$, $v_n<u_n$ in $\mathcal{D}$. 
			Hence, the Monotone Convergence Theorem and a barrier argumetn as in the Case $\left\{A_i\right\}=\emptyset$ imply that there exist $v=\lim_{n\to\infty}v_n>u=\lim_{n\to\infty}{u_n}_{|\mathcal{D}}$. In particular, we get that both $u$ and $v$ satisfy $P_{JS}(\mathcal{D})$ and $u=v$ since the solution is unique.
			\item Case $\left\{C_i\right\}=\emptyset$. The proof of this last case is identical to the one of \cite[Theorem 7.12]{HauRoSpr09}.
		\end{itemize}
		
		Step (2). 
		Since the proof of this step is analogous to the one of \cite[Theorems 3.1 and 3.2]{FolMe12}, we recall just the main arguments, avoiding the details.
		
		Using the Compactness Theorem and a diagonal argument as in the proof of \Cref{thm:entire}, we are able to construct the sequence $\left\{u_n\right\}$ of solution of 
		\[\begin{cases}
			\begin{array}{lr}
				H(u_n)=H&\textbf{in }\mathcal{D}^*;\\
				u_n=f_i&\textbf{in }C_i;\\
				u_n=n&\textbf{in }A_i;\\
				u_n=-n&\textbf{in }B_i^*.
			\end{array}
		\end{cases}\]
		As in the second case of the previous step it is easy to see that $\left\{u_n\right\}$ diverges to $-\infty$ in the domain bounded by $B_i\cup B_i^*$, for all $i$, and in particular $\left\{B_i\right\}\subset\left\{L_i\right\}$.
		Now, we distinguish two cases.
		\begin{itemize}
			\item Case $\left\{C_i\right\}=\emptyset$.
			\begin{itemize}
				\item Since $\partial_\infty\mathcal{D}$ has a finite number of vertices, we can suppose that the divergence set is composed of a finite number of disjoint divergence lines. These lines separate
				the domain $\mathcal{D}$ in at least two connected components, and the interior of these components belongs to the convergence domain. The aim is to prove that $\left\{L_i\right\}=\left\{B_i\right\}$.
				\item If $\mathcal{U}'\subset\mathcal{U}=\mathcal{D}\setminus\left\{L_i\right\}$, for any $p\in\mathcal{U}'$, $\left\{u_n-u_n(p)\right\}$ converges in compact subsets of $\mathcal{U}'$ and $\partial\mathcal{U}'$ is an inscribed $H$-polygon.
				\item Without loss of generality there exists only one divergence line $L_{\mathcal{U}'}\subset\partial\mathcal{U}'$ and that it is concave with respect to $\mathcal{U}'$ (if it is convex the argument is analogous after changing the sign). In particular, if we denote by $u_p=\lim_{n\to\infty}u_n-u_n(p)$, for any arc $\Gamma\subset L_{\mathcal{U}'}$ we get that $\Flux{u}{\Gamma}=-|\Gamma|$.
				\item If we denote by $H_i(m)$ the horocycle at the ideal vertex $p_i$ with euclidean radius $1/m$, we get that for $m $ sufficiently large $\partial\mathcal{U}'$ intersects each $H_i$ in two points.
				So, we can define $\mathcal{P}(m)=\partial\left(\mathcal{U}'\setminus \cup H_i(m)\right)$.
				\item Using \Cref{lemma:flux0}, one gets
				\[\begin{array}{rl}
					2H\mathcal{A}\left(\mathcal{U}'\setminus\cup H_i(m)\right)=&\Flux{u}{\mathcal{P}(m)}\\
					=&\Flux{u}{\mathcal{P}(m) \setminus\left(\bigcup_i(A_i\setminus H_j(m)\cup \bigcup_i(H_i\cap\mathcal{U}'))\right)}\\&\qquad\qquad\qquad+\Flux{u}{\left(\bigcup_i(A_i\setminus H_j(m)\cup \bigcup_i(H_i\cap\mathcal{U}'))\right)}\\
					\leq& 2\alpha(\mathcal{\partial\mathcal{U}'})-\gamma(\mathcal{\partial\mathcal{U}'})+\sum_i|H_i(m)\cap\mathcal{U}'|
				\end{array}\]
				and taking the limit for $m\to\infty$, one gets
				\[2H\mathcal{A}\left(\mathcal{U}'\right)\leq2\alpha\left(\partial\mathcal{U}'\right)-\gamma\left(\partial\mathcal{U}'\right)\] contradicting \eqref{eq:JScond}.
				\item Inductively, one gets that $\left\{L_i\right\}=\left\{B_i\right\}$. So we take $p\in\mathcal{U}'=\mathcal{D}$ and consider the sequence $u_n-u_n(p)$ that converges to the solution $u$.
				\item If $\lim\, u_n(p)$ is bounded, then $u$ easily takes the desired boundary values.
				\item If $\lim\, u_n(p)=-\infty$, then $u$ diverges to $+\infty$ on $\bigcup_iA_i$. By the flux formula, one gets
				\[\begin{array}{rl}
					\lim_{n\to\infty}\Flux{u_n}{\mathcal{P}(m)}=&2H\mathcal{A}(\mathcal{D})-2H\mathcal{A}(\mathcal{D}\cap(\cup F_i(m)))\\
					=& \sum\lim_{n\to\infty}\Flux{u_n}{A_i(m)}+\sum\lim_{n\to\infty}\Flux{u_n}{B_i(m)}\\&\qquad\qquad\qquad\qquad\qquad\qquad+\sum\lim_{n\to\infty}\Flux{u_n}{\mathcal{D}\cap H_i(m)}\\
					\geq& \alpha(\partial\mathcal{D})-\beta(\partial\mathcal{D})-\sum|\mathcal{D}\cap H_i(m)|,
				\end{array}\]
				which implies
				\[2\beta(\mathcal{\partial\mathcal{D}})\geq\gamma(\mathcal{\partial\mathcal{D}})-2H\mathcal{A}(\mathcal{D}).\]
				The strict inequality is not allowed by the hypothesis and so $2\beta(\mathcal{\partial\mathcal{D}})=\gamma(\mathcal{\partial\mathcal{D}})-2H\mathcal{A}(\mathcal{D}).$ This implies that $$\lim_{n\to\infty}\Flux{u_n-u_n(p)}{B_i(m)}=-|B_i(m)|$$ and then $u$ diverges to $-\infty$ on $\bigcup_iB_i$.
				\item  If $\lim\, u_n(p)=+\infty$, then $u$ diverges to $-\infty$ on $\bigcup_iB_i$ and we can argue as before to prove that $u$ diverges to $+\infty$ on $\bigcup_iA_i$.
				\item To prove the necessity of conditions \eqref{eq:JScond0} and \eqref{eq:JScond} we suppose that there exists a solution $u$ in $\mathcal{D}$. Applying the flux formula to $\mathcal{P}(m)=\partial(\mathcal{D}\setminus\cup F_i(m))$, one gets
				\[\begin{array}{rl}
					2H\mathcal{A}(\mathcal{D}\setminus\cup F_i(m))=&\Flux{u}{\mathcal{P}(m)}\\
					=&\Flux{u}{A_i(m)}+\Flux{u}{B_i(m)}+\Flux{u}{ H_i(m)\cap\mathcal{D}}
				\end{array}\]
				In particular,
				\[\begin{array}{c}
					\sum|A_i(m)|-\sum|B_i(m)|-\sum| H_i(m)\cap\mathcal{D}|\leq 2H\tilde{\mathcal{A}}(\mathcal{D})-2H\mathcal{A}(\tilde{\mathcal{D}})\leq\sum|A_i(m)|-\sum|B_i(m)|+\sum| H_i(m)\cap\mathcal{D}|.
				\end{array}\]
				Taking the limit for $m\to\infty$, one gets $| H_i(m)\cap\mathcal{D}|\to 0$ and $\mathcal{A}(\tilde{\mathcal{D}})\to0$, so $\alpha(\partial\mathcal{D})-\beta(\partial\mathcal{D})=2H\tilde{\mathcal{A}}(\mathcal{D})$ and it shows the necessity of \eqref{eq:JScond0}.
				
				To prove the necessity of \eqref{eq:JScond}, we consider any $H$-polygon $\mathcal{P}\subset\mathcal{D}$ and we denote by $E_i$ the arcs of $\mathcal{P}$ in the interior of $\mathcal{D}$. So, we apply again the flux formula
				\[\begin{array}{rl}
					2H\mathcal{A}(\mathcal{D}_\mathcal{P}^{\mathcal{H}(m)})=&\Flux{u}{\mathcal{P}(m)}\\
					=&\sum_k\Flux{u}{A_k(m)}+\sum_j\Flux{u}{B_j(m)}+\sum_l\Flux{u}{E_l(m)}+\sum_i\Flux{u}{H_i(m)\cap\mathcal{D}_\mathcal{P}}\\
					\geq&\sum_k|A_k(m)|-\sum_j|B_j(m)|-\sum_l|E_l(m)|-\sum_i|H_i(m)\cap\mathcal{D}_\mathcal{P}|+\delta\\
					=&2\alpha(\mathcal{P})-\gamma(\mathcal{P})-\sum_i|H_i(m)\cap\mathcal{D}_\mathcal{P}|+\delta.
				\end{array}\]
				Since $\tilde{\mathcal{A}}(\mathcal{D}_\mathcal{P})>\mathcal{A}(\mathcal{D}_\mathcal{P}^{\mathcal{H}(m)})$ and $\sum_i|H_i(m)\cap\mathcal{D}_\mathcal{P}|-\delta<0$ for $n$ large enough, so
				\[2\alpha(\mathcal{P})<\gamma(\mathcal{P})+2H\tilde{\mathcal{A}}(\mathcal{D}_\mathcal{P}).\]	
				
				Similarly,
				\[\begin{array}{rl}
					2H\mathcal{A}(\mathcal{D}_\mathcal{P}^{\mathcal{H}(m)})=&\Flux{u}{\mathcal{P}(m)}\\
					=&\sum_k\Flux{u}{A_k(m)}+\sum_j\Flux{u}{B_j(m)}+\sum_l\Flux{u}{E_l(m)}+\sum_i\Flux{u}{H_i(m)\cap\mathcal{D}_\mathcal{P}}\\
					\leq&\sum_k|A_k(m)|-\sum_j|B_j(m)|+\sum_l|E_l(m)|+\sum_i|H_i(m)\cap\mathcal{D}_\mathcal{P}|-\delta\\
					=&-2\beta(\mathcal{P})+\gamma(\mathcal{P})+\sum_i|H_i(m)\cap\mathcal{D}_\mathcal{P}|-\delta.
				\end{array}\]
				Then,  for $n>>0$,
				\[2\beta(\mathcal{P})<\gamma(\mathcal{P})-2H\tilde{\mathcal{A}}(\mathcal{D}_\mathcal{P}).\]	
			\end{itemize}

			\item Case $\left\{C_i\right\}\neq\emptyset$. The proof of this case is similar to the previous one. We refer to the proof of \cite[Theorem 3.2]{FolMe12} for the details.
		\end{itemize}
		
		Step (3). We use an idea similar to the one of the proof of \Cref{thm:entire} to construct the sequence that will converge to our solution.
		For any $D_i$ let $\left\{D_i^n\right\}$ a family of curves of constant geodesic curvature $\kappa_g(D_i^n)=\tanh(n)$, joining the endpoints of $D_i$. Denote by $\mathcal{D}^n\subset\mathcal{D}$ the domain bounded by $A_i\cup B_i\cup C_i\cup D_i^n$ (this is well defined for every $n\geq n_0$, where $n_0$ is the smaller natural number such that $D_i^n\cap\left(A_i\cup B_i\cup C_i\right)=\emptyset$). Let \[\eta\colon\mathbb{B}((0,0),1)\setminus\left\{(0,0)\right\}\to\partial\mathbb{B}((0,0),1),\qquad\eta(x,y)=\left(\frac{x}{\sqrt{x^2+y^2}},\frac{y}{\sqrt{x^2+y^2}}\right),\] where $\mathbb{B}((0,0),1)$ is the euclidean ball centered in $(0,0)$ with radius $1$, and denote by $g_i^n= g_i\in\mathcal{C}(D_i^n)\circ \eta_{|D_i^n}$. From the previous step we know that for each $n>\max\left\{n_0,n_1=\arctanh(2|H|)\right\}$ there exists a solution $u_n$ of the Jenkins-Serrin problem $P_{JS}(\mathcal{D}^n)$ such that $u_n=g_i^n$ in $D_i^n$.
		 {For each $D_i$ we denote by $\gamma_i^{\pm}$ the curve of constant geodesic curvature $\pm2H$ joining the endpoints of $D_i$ and by $\mathcal{D}_i^\pm$ the domain bounded by $\gamma_i^\pm\cup D_i$ such that $\mathcal{D}_i^+$ in convex and $\mathcal{D}_i^-$ is concave. And denote by $\mathcal{D}^{up}$ the domain bounded by $A_i\cup B_i\cup C_i\cup \gamma_i^-$. Condition (D) of the admissibility of the Jenkins--Serrin domain guaranty that $\mathcal{D}^{up}$, while condition (C) and the previous step guarantee the existence of a unique solution to the following Jenkins--Serrin problem in $\mathcal{D}^{up}$:
		\begin{equation*}
			\begin{cases}
				\begin{array}{ll}
					H(u^{up})=H&\textrm{in }\mathcal{D}^{up},\\
					u^{up}=+\infty&\textrm{in }\cup A_i\cup\gamma_i^-,\\
					u^{up}=0&\textrm{in }\cup B_i,\\
					u^{up}=f_i&\textrm{in } C_i.
				\end{array}
			\end{cases}
		\end{equation*}
		The Maximum Principle implies that $u^{up}>u_n$ for any $n$ in $\mathcal{D}^{up}$. Now for each $D_i$ let $\gamma_i^*\subset\mathcal{D}^{up}$ be a piecewise smooth curve joining the endpoints of $D_i$ with constant geodesic curvature $2H$ with respect to the normal pointing to $D_i$ and such that $u^{up}_{|\gamma_i^*}$ is bounded and denote by $\mathcal{D}_i^*$ the domain bounded by $D_i\cup\gamma_i^*$. 
		The Maximum Principle implies that in $\mathcal{D}_i^*\cap\mathcal{D}^n$, for every $i,n$, \[\max\left\{\max_{\gamma_i^*}u^{up},\max_{D_i}g_i\right\}>u_n.\]  In particular, we have found an upper barrier for the sequence $\left\{u_n\right\}$ in every compact subset of $\mathcal{D}$.}
		
		 {To find a lower barrier bound for $\left\{u_n\right\}$ in every compact subset of $\mathcal{D}$ we notice that in each $\mathcal{D}_i^-$ there exists a function $\omega^-_i$ (described in \Cref{sec:Barriers}) whose graphs is the $H$-surface invariant by hyperbolic translation such that $\omega^-_i$ diverges to $-\infty$ approaching $\gamma_i^-$ and is equal to the function $f$ (bounded above) in $D_i$. In particular, up to a vertical translation we can assume that $f<g_i$ in $D_i$, and so the Maximum Principle implies that $\omega^-_i<u_n$ in $\mathcal{D}_i^-$ for every $n$. Furthermore, fixed $n^*$ sufficiently large such that $\mathcal{D}^{n^*}$ is well defined, the previous step guarantees that there exists a solution to 
		\[\begin{cases}
			\begin{array}{ll}
				H(u^{down})=H&\textrm{in }\mathcal{D}^{n^*},\\
				u^{down}=0&\textrm{in }\cup A_i,\\
				u^{down}=-\infty&\textrm{in }\cup B_i,\\
				u^{down}=f_i&\textrm{in } C_i,\\
				u^{down}=\omega^-_i&\textrm{in } D_i^n.
			\end{array}
		\end{cases}\]
		and, by applying the Maximum Principle, it follows that $u^{down}<u_n$ in $\mathcal{D}^{n^*}$, for every $n>n^*$.}
		 
		Now, the Compactness Theorem implies that there exists a subsequence of $\left\{u_n\right\}$ that converges to a solution of the constant mean curvature equation equal to $H$ in $\mathcal{D}$ with the right boundary values in all the $A_i,B_i,C_i$ by a standard barrier argument. The same barrier argument used in \Cref{thm:entire} guarantees that $u_{|D_i}=g_i$. The Maximum Principle gives the uniqueness and conclude the proof.
	\end{proof}

\end{document}